%% file: HSDG.tex
\documentclass[10pt]{amsart}
\input{./mysetting.tex}
\usepackage{bm}
\usepackage{enumitem}
\newcommand{\sym}{\operatorname{sym}}
\newcommand{\Otimes}{\ensuremath{\vcenter{\hbox{\scalebox{1.5}{$\otimes$}}}}}
\newcommand{\Oplus}{\ensuremath{\vcenter{\hbox{\scalebox{1.5}{$\oplus$}}}}}

\begin{document}
\title[SDG for polyharmonic equations on polytopes]
{Hybridizable Staggered Discontinuous Galerkin Methods for Polyharmonic Equations on Polytopes}

\author{Long Chen}
\address{Department of Mathematics, University of California, Irvine, CA 92697, USA}
\email{chen-long@math.uci.edu}

\author{Xuehai Huang}
\address{School of Mathematics, Shanghai University of Finance and Economics, Shanghai 200433, China}
\email{huang.xuehai@sufe.edu.cn}

\author{Yule Sun}
\address{School of Mathematics and Computational Science, Xiangtan University, Xiangtan 411105, China}
\email{sunyuleqn@foxmail.com}

\author{Shudan Tian}
\address{School of Mathematics and Computational Science, Xiangtan University, Xiangtan 411105, China}
\email{shudan.tian@xtu.edu.cn}

\thanks{The first author was partially supported by NSF grant DMS-2309785, the third author was supported by NSFC grant 12431014, and the fourth author was supported by NSFC grant 12401483.}

\makeatletter
\@namedef{subjclassname@2020}{\textup{2020} Mathematics Subject Classification}
\makeatother
\subjclass[2020]{
65N30;   
65N12;   
15A72.   
}

\begin{abstract}
Hybridizable staggered discontinuous Galerkin methods are developed for arbitrary-order polyharmonic equations $(-\Delta)^m u=f$ on shape-regular polytopal meshes in $\mathbb R^d$, for any $m\ge1$, $d\ge2$, and polynomial degree $k\ge0$. The method uses the mixed variable $\sigma=\nabla^m u$ and a staggered primal--dual mesh to impose complementary continuity on scalar and tensor unknowns, without restrictions such as $d\ge m$. Local trace and bubble enrichments stabilize low-order tensor spaces without adding global unknowns. Hybridization localizes the tensor variable and yields an equivalent stabilization-free weak Galerkin formulation. Well-posedness and optimal energy error estimates are proved, and numerical experiments on polygonal and tetrahedral meshes confirm the predicted rates.
\end{abstract}

\maketitle


\section{Introduction}

Let $\Omega$ be a bounded polyhedral domain in $\mathbb R^d$, $d\ge2$. In this paper, we develop hybridizable staggered discontinuous Galerkin (HSDG) methods for the $2m$-th order polyharmonic problem, for $m\ge1$,
\begin{equation}\label{intro:polyharmonic}
\begin{cases}
(-1)^m\Delta^m u=f, & \textrm{in }\Omega,\\
u=\partial_n u=\cdots=\partial_n^{m-1}u=0, & \textrm{on }\partial\Omega.
\end{cases}
\end{equation}
The method is designed for general polytopal meshes, arbitrary order $m\geq 1$, arbitrary dimension $d\geq 2$, and all polynomial degrees $k\ge0$. Its main ingredients are a staggered primal--dual mesh, a mixed tensor formulation, local tensor enrichments for low-order spaces, and a hybridization procedure that reduces the global system to scalar and face unknowns.

The starting point is the mixed formulation
\begin{equation}\label{intro:mixfem}
\sigma=\nabla^m u,\qquad
(-1)^m\operatorname{div}^m\sigma=f .
\end{equation}
Formally, the relation $\sigma=\nabla^m u$ can be written as
\begin{equation*}
(\sigma,\tau)=(-1)^m(\operatorname{div}^m\tau,u),
\end{equation*}
so the scalar unknown only needs to be in $L^2(\Omega)$, while the tensor unknown carries the $H(\operatorname{div}^m)$-type regularity. For $m=1$, this is the classical mixed formulation based on $H(\operatorname{div})$-conforming spaces, including Raviart--Thomas, Brezzi--Douglas--Marini, Nédélec, and related elements~\cite{RaviartThomas1977,BrezziDouglasMarini1986,BrezziDouglasDuranFortin1987,Nedelec:1986family,Nedelec1980}. For $m=2$, several $H(\operatorname{div}\operatorname{div})$-conforming finite element spaces have recently been developed~\cite{ChenHuang2025div-div-conforming,ChenHuang2022,ChenHuang2022b,ChenHuang2020,Hu;Ma;Zhang:2020family}. For general $m$, constructing explicit $H(\operatorname{div}^m)$-conforming tensor finite elements is difficult, especially on general polytopes.

The staggered discontinuous Galerkin (SDG) framework provides a different way to impose the mixed regularity. SDG methods, initiated by Chung and Engquist~\cite{ChungEngquist06,ChungWave09}, use staggered primal and dual grids to obtain stable and locally conservative discretizations. They have been extended from triangular meshes to polygonal meshes in~\cite{LinaParkdiffusion18,LinaParkelasticity20,ZhaoParkShin19}, and a primal SDG method on polytopal meshes was recently developed in~\cite{ChenHuangParkWang2025}. The present work extends~\cite{ChenHuangParkWang2025} from second-order problems to arbitrary-order polyharmonic equations.

Given a polyhedral primal mesh $\mathcal K_h$, each element $K$ is split by connecting an interior point of $K$ to its faces, giving a refined mesh $\mathcal K_h^{\rm R}$. The dual mesh $\mathcal K_h^*$ is formed by grouping refined cells that share a primal face. The key SDG idea is to impose complementary continuity on the primal and dual meshes. The tensor variable $\sigma_h$ satisfies the required $H(\operatorname{div}^m)$-type conformity on dual elements, while the scalar variable $u_h$ is locally $H^m$-conforming on primal elements. Thus both variables are globally discontinuous, but their continuity properties are staggered and complementary. This gives a stable mixed discretization of \eqref{intro:mixfem} without constructing globally $H(\operatorname{div}^m)$-conforming tensor elements on the primal polytopal mesh.

The scalar space is chosen elementwise as $\mathbb P_{k+m}(K)$ on each primal element. The main difficulty is to construct tensor spaces $\Sigma_{h,k}$ for an arbitrary number of face-normal directions, order $m$, dimension $d$, and polynomial degree $k$. Two local enrichments are used. The trace enrichment recovers missing normal trace layers when $k<m-1$. The bubble enrichment provides enough tensor directions for the local stability construction. If $\nu_K$ is the largest number of pairwise non-parallel face normals of $K$, then tensor layers up to order $\lfloor m/\nu_K\rfloor$ are enough.

The method is naturally hybridizable. Following the hybridization framework for mixed methods~\cite{arnold1985mixed,cockburn2009unified}, we relax the primal-face continuity of the tensor traces, use a fully broken tensor space on the refined mesh, and introduce scalar trace unknowns on primal faces. The weak $m$-th gradient $\nabla^m_w$ is defined so that the tensor unknown can be eliminated element by element. After static condensation, the global problem involves only the scalar variable and its face traces and is equivalent to a weak Galerkin formulation without penalty stabilization. Stability comes from the staggered mixed structure and the local tensor trace construction, which yield coercivity of the discrete weak gradient after hybridization.


We now place the proposed method in context. The primal formulation of \eqref{intro:polyharmonic} uses the bilinear form $(\nabla^m u,\nabla^m v)$ and therefore requires an $H^m$-conforming discrete space. On simplicial meshes, this leads to highly smooth finite elements, such as those in~\cite{bramble_triangular_1970,vzenivsek1974tetrahedral,Lai;Schumaker:2007Trivariate,zhang_family_2009,HuLinWu2024}. A geometric decomposition of the simplicial lattice was introduced in~\cite{chen_geometric_2021,Chen;Huang:2022FEMcomplex3D} and used for explicit basis construction and implementation in~\cite{ChenChenGaoHuangEtAl2025}. These conforming finite element spaces are explicit, but they require high polynomial degree $2^d (m-1)+1$.

On general polytopes, conforming virtual element methods provide a flexible alternative. The $H^m$-conforming virtual elements of arbitrary degree $k\ge m$ on polytopal meshes were constructed in~\cite{ChenHuangWei2022}, generalizing earlier two- and three-dimensional constructions~\cite{BeiraodaVeigaManzini2014,AntoniettiManziniVerani2020,AntoniettiManziniScacchiVerani2021}. These methods work on arbitrary polytopes and arbitrary degree, but the local shape functions are virtual and a nontrivial stabilization term is needed.

Nonconforming methods reduce the smoothness requirement. On simplicial meshes, minimal $H^m$-nonconforming elements were constructed for $m\le d$~\cite{ming2006morley,wang2013minimal} and later extended to $m=d+1$~\cite{wu2019nonconforming}; arbitrary $m$ and $d$ can also be handled by interior penalty techniques~\cite{wu2017mathcal}. The minimal nonconforming elements correspond to the lowest-order case $k=0$ in our notation. A related minimal bubble-enriched construction was proposed in~\cite{Wuli-nonconforming}.
On general polytopes, $H^m$-nonconforming virtual elements of arbitrary degree were developed in~\cite{ChenHuang2020a,Huang2020}. These VEMs, however, require nontrivial stabilization. The present paper follows a different route: the primary object is a staggered mixed DG method, and its stability is obtained from complementary primal--dual continuity and tensor trace liftings, not from a primal nonconforming space with a stabilization term.

After local tensor elimination, the hybridized SDG method can be interpreted as a weak Galerkin or VEM-like method without penalty stabilization. In this interpretation, the scalar unknown and its face traces determine a local weak $m$-th gradient. Unlike standard VEM projections on a single polytopal element, the projection here is enriched from a polynomial space on $K$ to a piecewise polynomial space on the barycentric split $K^{\rm R}$. This larger local projection space replaces penalty stabilization.

The analysis is carried out under standard shape-regularity assumptions on the polytopal mesh and on the associated refined subcells. The constants in the stability and error estimates are uniform with respect to the mesh size $h$ and may depend only on the fixed polynomial degree, the order $m$, the dimension $d$, and the mesh regularity parameters. These assumptions are stated precisely in Section~\ref{sec:mesh}.

The remainder of this article is organized as follows. Section~\ref{sec:symtensor} introduces the notation, tensor calculus, and staggered mesh structure. Section~\ref{sec:femsymtensor} proves the tensor polynomial decompositions and local trace constructions. Section~\ref{sec:sdg} constructs the SDG finite element spaces and proves the well-posedness of the mixed method. Section~\ref{sec:hybrid} develops the hybridized SDG formulation and proves the error estimates. Section~\ref{sec:numerexperiments} reports numerical experiments supporting the theoretical results.

\section{Symmetric Tensors and Decompositions}\label{sec:symtensor}

In this section, we recall basic notation for tensor spaces, contractions, tensor differential operators, and symmetric tensors. We then use simplicial lattices to index bases of symmetric tensor spaces and to describe polynomial layers relative to a face. Finally, we derive a geometric $t$-$n$ decomposition of symmetric tensors and a face-based spanning decomposition that will be used in the construction of the local tensor spaces.

\subsection{Tensors}

Denote the space of $d$-dimensional tensors of order $m$ by $\mathbb R^{d,m}:=(\mathbb R^d)^{\otimes m}.$
Let $\{\boldsymbol e_1,\ldots,\boldsymbol e_d\}$ be an orthonormal basis of $\mathbb R^d$. Then
$$
\left\{
\boldsymbol e_{i_1}\otimes\boldsymbol e_{i_2}\otimes\cdots\otimes\boldsymbol e_{i_m}
: i_\ell\in\{1,\ldots,d\}, \ \ell=1,\ldots,m
\right\}
$$
is a basis of $\mathbb R^{d,m}$. Hence any $\tau\in\mathbb R^{d,m}$ can be written as
$$
\tau
=
\tau_{i_1,\ldots,i_m}
\boldsymbol e_{i_1}\otimes\boldsymbol e_{i_2}\otimes\cdots\otimes\boldsymbol e_{i_m},
$$
where repeated indices are summed over $1,\ldots,d$. Therefore $\dim\mathbb R^{d,m}=d^m.$ Equivalently, the basis can be represented by the integer lattice $\{1,\ldots,d\}^m$; see Fig.~\ref{fig:tensor}(a).

Let $\tau\in\mathbb R^{d,m}$ and $\gamma\in\mathbb R^{d,n}$ with $0\le n\le m$. The contraction
$$
\tau\mathbin{\lrcorner}\gamma\in\mathbb R^{d,m-n}, \quad (\tau\mathbin{\lrcorner}\gamma)_{i_1,\ldots,i_{m-n}}
:=
\tau_{i_1,\ldots,i_{m-n},j_1,\ldots,j_n}
\gamma_{j_1,\ldots,j_n}.
$$
In particular, if $n=1$, then $\tau\mathbin{\lrcorner}\gamma$ is the tensor-vector contraction, extending the matrix-vector product. If $n=m$, then
$$
\tau:\gamma
:=
\tau\mathbin{\lrcorner}\gamma
=
\tau_{j_1,\ldots,j_m}\gamma_{j_1,\ldots,j_m}
$$
defines the Euclidean inner product on $\mathbb R^{d,m}$, and the induced norm is the Frobenius norm.

For $\tau\in\mathbb R^{d,m}$ and $\boldsymbol v\in\mathbb R^d$, the outer product $\tau\otimes\boldsymbol v\in\mathbb R^{d,m+1}$ is defined by
$$
(\tau\otimes\boldsymbol v)_{i_1,\ldots,i_m,j}
=
\tau_{i_1,\ldots,i_m}v_j.
$$

Let $\Omega\subset\mathbb R^d$ be a domain. A tensor function on $\Omega$ is a tensor whose coefficients are functions on $\Omega$. For $s\ge0$, define $H^s(\Omega;\mathbb R^{d,m})$ as the space of tensor functions whose coefficients belong to $H^s(\Omega)$. Let $\mathbb P_k(T)$ be the space of polynomials of degree at most $k$ on $T$, with the convention $\mathbb P_k(T)=\{0\}$ for $k<0.$
For a finite-dimensional linear space $V$, define
$$
\mathbb P_k(T;V):=\mathbb P_k(T)\otimes V.
$$
Equivalently, $\mathbb P_k(T;V)$ is the space of $V$-valued polynomials of degree at most $k$ on $T$.

For $\tau\in H^1(\Omega;\mathbb R^{d,m})$, the divergence is the tensor in $L^2(\Omega;\mathbb R^{d,m-1})$ defined by
$$
(\operatorname{div}\tau)_{i_1,\ldots,i_{m-1}}
=
\partial_j\tau_{i_1,\ldots,i_{m-1},j}.
$$
For $0\le s\le m$, let $\operatorname{div}^s$ denote the operator obtained by applying divergence $s$ times. With $\nabla=(\partial_1,\ldots,\partial_d)^{\intercal}$, this can be written symbolically as
\begin{equation}\label{eq:tensor-div-symbolic}
\operatorname{div}^{s}\tau
=
\tau\mathbin{\lrcorner}\nabla^s .
\end{equation}

\begin{figure}[htbp]
\centering
\subfigure[The cubic lattice $(1:3)^3$ represents the tensor space $\mathbb R^{3,3}$, which has dimension $27$.]{
\begin{minipage}[t]{0.425\linewidth}
\centering
\includegraphics*[width=3.5cm]{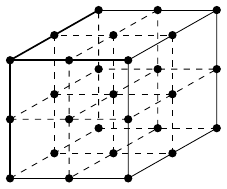}
\end{minipage}}%
\qquad
\subfigure[The simplicial lattice $\mathbb T^2_3=\{\alpha\in\mathbb N^3:\alpha_0+\alpha_1+\alpha_2=3\}$ represents the symmetric tensor space $\mathbb S^{3,3}$, which has dimension $10$.]{
\begin{minipage}[t]{0.475\linewidth}
\centering
\includegraphics*[width=3.2cm]{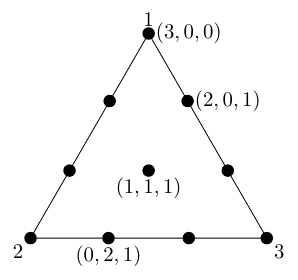}
\end{minipage}}
\caption{Lattice representations of tensor and symmetric tensor spaces.}
\label{fig:tensor}
\end{figure}

\subsection{Symmetric tensors}

Denote by $\mathbb S^{d,m}$ the space of $d$-dimensional symmetric tensors of order $m$. It is the subspace of $\mathbb R^{d,m}$ consisting of tensors $\tau$ such that
$$
\tau_{i_1,\ldots,i_m}
=
\tau_{i_{\sigma(1)},\ldots,i_{\sigma(m)}}
\qquad
\text{for all }\sigma\in\mathfrak S_m,
$$
where $\mathfrak S_m$ is the symmetric group on $\{1,\ldots,m\}$.

The symmetrization operator $\sym:\mathbb R^{d,m}\to\mathbb S^{d,m}$ is defined by
$$
\sym(\boldsymbol v_1\otimes\cdots\otimes\boldsymbol v_m)
=
\frac{1}{m!}
\sum_{\sigma\in\mathfrak S_m}
\boldsymbol v_{\sigma(1)}
\otimes\cdots\otimes
\boldsymbol v_{\sigma(m)},
\qquad
\boldsymbol v_1,\ldots,\boldsymbol v_m\in\mathbb R^d,
$$
and extended by linearity.

For $v\in H^m(\Omega)$, define the $m$-th gradient by
$$
\nabla^m v
:=
\nabla^{\otimes m}v
=
(\underbrace{\nabla\otimes\cdots\otimes\nabla}_{m})v, \quad 
(\nabla^m v)_{i_1,\ldots,i_m}
=
\partial_{i_1}\cdots\partial_{i_m}v.
$$
If $v$ is sufficiently smooth, mixed partial derivatives commute. Hence $\nabla^m v\in\mathbb S^{d,m}.$

Let $\alpha=(\alpha_1,\ldots,\alpha_d)\in\mathbb N^d$ be a multi-index, with $|\alpha|=\sum_{i=1}^d\alpha_i$. For $v\in C^m$, the $m$-th derivatives can also be indexed by
$$
D^\alpha v
:=
\frac{\partial^m v}
{\partial x_1^{\alpha_1}\cdots\partial x_d^{\alpha_d}},
\qquad
|\alpha|=m.
$$
This multi-index notation gives a simplicial lattice representation of symmetric tensors.

\subsection{Simplicial lattice representation}

A simplicial lattice~\cite{chen_geometric_2021,Chen;Huang:2021Geometric} of degree $r$ and dimension $d$ is
$$
\mathbb T^{d}_r
=
\left\{
\alpha=(\alpha_0,\alpha_1,\ldots,\alpha_d)\in\mathbb N^{d+1}
:\alpha_0+\alpha_1+\cdots+\alpha_d=r
\right\}.
$$
Each $\alpha\in\mathbb T^d_r$ is called a node of the lattice. Figure~\ref{fig:tensor}(b) shows $\mathbb T^2_3$.

For the symmetric tensor $\sym(\boldsymbol e_{i_1}\otimes\cdots\otimes\boldsymbol e_{i_m})$, the order of the factors in
$\boldsymbol e_{i_1}\otimes\cdots\otimes\boldsymbol e_{i_m}$ is irrelevant up to permutation. Only the number of occurrences of each basis vector is relevant. Let $\alpha_k$ be the number of factors equal to $\boldsymbol e_k$ for $k=1,\ldots,d$. Then $\alpha=(\alpha_1,\ldots,\alpha_d)$ satisfies $|\alpha|=m$. This set of multi-indices is identified with the lattice $\mathbb T^{d-1}_m$. Therefore
$$
\left\{
\sym\!\left(\Otimes_{k=1}^d \boldsymbol e_k^{\alpha_k}\right)
:\alpha\in\mathbb T^{d-1}_m
\right\},
$$
where $\boldsymbol e_k^{\alpha_k}:=\boldsymbol e_k^{\otimes\alpha_k}$, is a basis of $\mathbb S^{d,m}$. Hence
$$
\dim\mathbb S^{d,m}
=
|\mathbb T^{d-1}_m|
=
\binom{m+d-1}{m},
$$
which is much smaller than $\dim\mathbb R^{d,m}=d^m$.

\subsection{Layer decomposition of the simplicial lattice}

Let $F$ be a face of a $d$-simplex $T$. We label the barycentric coordinate opposite to $F$ by index $0$. Then each lattice point $\alpha\in\mathbb T^d_k$ can be written as
$$
\alpha
=
(\alpha_0,\alpha_1,\ldots,\alpha_d)
=
\alpha_0\boldsymbol e_0+(0,\alpha_F),
$$
where $\boldsymbol e_0=(1,0,\ldots,0)$ and
$\alpha_F=(\alpha_1,\ldots,\alpha_d)\in\mathbb T^{d-1}_{k-\alpha_0}.$
Define the distance of $\alpha$ to $F$ by
$$
\operatorname{dist}(\alpha,F):=\alpha_0,
\qquad
0\le \alpha_0\le k.
$$
For $0\le j\le k$, define the $j$-th layer by
$$
L_k^d(F,j)
:=
\left\{
j\boldsymbol e_0+(0,\alpha_F):
\alpha_F\in\mathbb T^{d-1}_{k-j}
\right\}.
$$
Thus $L_k^d(F,j)$ is the set of lattice points in $\mathbb T^d_k$ at distance $j$ from $F$. The union of the layers $L_k^d(F,j)$ for $j=\ell,\ldots,k$ is a translated copy of the smaller lattice $\mathbb T^d_{k-\ell}$; see Fig.~\ref{fig:tensor11}(a).

\begin{lemma}\label{lm:layerdecomposition}
The simplicial lattice admits the layer decomposition, for $0\le \ell\le k$,
\begin{equation}\label{eq:latticedecomposition}
\mathbb T^d_k
=
\bigsqcup_{j=0}^{k} L_k^d(F,j)
=
\bigsqcup_{j=0}^{\ell-1}L_k^d(F,j)
\sqcup
(\ell\boldsymbol e_0+\mathbb T^d_{k-\ell}).
\end{equation}
\end{lemma}

\begin{proof}
Each $\alpha\in\mathbb T^d_k$ has a unique first component $\alpha_0=j$, with $0\le j\le k$. Hence $\alpha$ belongs to exactly one layer $L_k^d(F,j)$, which proves the first identity in \eqref{eq:latticedecomposition}. The remaining points after removing the first $\ell$ layers are precisely those with $\alpha_0\ge\ell$. Writing $\alpha=\ell\boldsymbol e_0+\beta$ gives $\beta\in\mathbb T^d_{k-\ell}$, and this proves the second identity.
\end{proof}

\begin{figure}[htbp]
\centering
\subfigure[$k\ge m$.]{
\begin{minipage}[t]{0.29\linewidth}
\centering
\includegraphics*[width=1.25in]{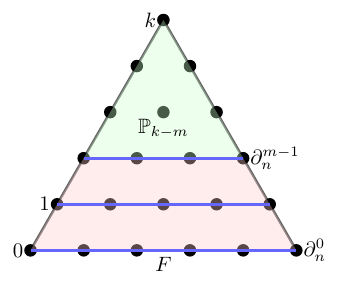}
\end{minipage}}%
\subfigure[$1\le k<m$.]{
\begin{minipage}[t]{0.29\linewidth}
\centering
\includegraphics*[width=1.15in]{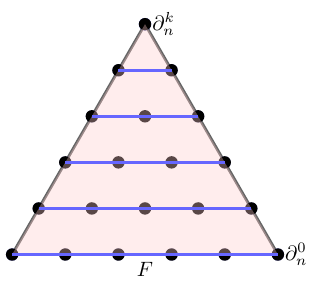}
\end{minipage}}%
\subfigure[$k=0$.]{
\begin{minipage}[t]{0.29\linewidth}
\centering
\includegraphics*[width=1.02in]{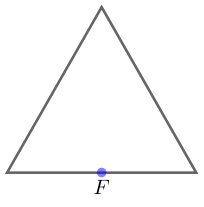}
\end{minipage}}
\caption{Layer decomposition of a two-dimensional simplicial lattice.}
\label{fig:tensor11}
\end{figure}

\subsection{Layer decomposition of polynomial spaces}
Let $T$ be a $d$-simplex with barycentric coordinates $\{\lambda_i, i=0,\ldots, d\}$. 
We will apply the layer decomposition \eqref{eq:latticedecomposition} to the Bernstein basis of $\mathbb P_k(T)$,
$$
\mathbb P_k(T)
=
\operatorname{span}\{\lambda^\alpha:=
\lambda_0^{\alpha_0}\lambda_1^{\alpha_1}\cdots\lambda_d^{\alpha_d} :\alpha\in\mathbb T^d_k\}.
$$
Since $\lambda_0$ is the barycentric coordinate opposite to $F$, the $j$-th lattice layer gives
$$
\operatorname{span}\{\lambda^\alpha:\alpha\in L_k^d(F,j)\}
=
\lambda_0^j\, \mathbb P_{k-j}(F).
$$

\begin{lemma}\label{lem:poly-layer-decomp}
The polynomial space $\mathbb P_k(T)$ has the layer decomposition, for $0\le i\le k$,
\begin{equation}\label{eq:poly-layer-decomp}
\mathbb P_k(T)
=
\Oplus_{j=0}^{k}
\lambda_0^j\, \mathbb P_{k-j}(F)
=
\left(
\Oplus_{j=0}^{i-1}
\lambda_0^j\, \mathbb P_{k-j}(F)
\right)
\Oplus
\lambda_0^i\, \mathbb P_{k-i}(T).
\end{equation}
\end{lemma}

\begin{proof}
The Bernstein basis functions are indexed by $\mathbb T^d_k$. By \eqref{eq:latticedecomposition}, each index $\alpha\in\mathbb T^d_k$ belongs to exactly one layer $L_k^d(F,j)$. If $\alpha\in L_k^d(F,j)$, then $\alpha_0=j$ and
$$
\lambda^\alpha
=
\lambda_0^j\lambda_F^{\alpha_F},
\qquad
\alpha_F\in\mathbb T^{d-1}_{k-j}.
$$
Thus the functions from this layer span $\lambda_0^j\, \mathbb P_{k-j}(F)$. Summing over $j=0,\ldots,k$ gives the first identity \eqref{eq:poly-layer-decomp}. The second identity follows by grouping all layers with $j\ge i$:
$$
\Oplus_{j=i}^{k}
\lambda_0^j\, \mathbb P_{k-j}(F)
=
\lambda_0^i
\Oplus_{j=0}^{k-i}
\lambda_0^{j}\, \mathbb P_{k-i-j}(F)
=
\lambda_0^i\, \mathbb P_{k-i}(T).
$$
The translated sublattice is illustrated in Fig.~\ref{fig:tensor11} (a). 
\end{proof}

\subsection{Layer decomposition of symmetric tensors}

We use the layer decomposition of the lattice $\mathbb T^{d-1}_m$ to decompose the symmetric tensor space $\mathbb S^{d,m}$ with respect to a face $F$.

Choose a local $t$-$n$ coordinate system associated with $F$,
$$
(\boldsymbol t_1,\ldots,\boldsymbol t_{d-1},\boldsymbol n_F),
$$
where $\boldsymbol t_1,\ldots,\boldsymbol t_{d-1}$ form an orthonormal basis of the tangent space of $F$ and $\boldsymbol n_F$ is a unit normal vector. 

For a tangential multi-index $\alpha_F=(\alpha_1,\ldots,\alpha_{d-1})$, define
$$
\boldsymbol t^{\alpha_F}
:=
\Otimes_{i=1}^{d-1}\boldsymbol t_i^{\alpha_i}.
$$
Using the $t$-$n$ basis, the symmetric tensor space is spanned by
$$
\mathbb S^{d,m}
=
\operatorname{span}
\left\{
\sym\big(
\boldsymbol n_F^\ell\otimes\boldsymbol t^{\alpha_F}
\big)
:
0\le\ell\le m,\ 
\alpha_F\in\mathbb T^{d-2}_{m-\ell}
\right\}.
$$
Define the $\ell$-th normal layer by
$$
\mathbb S^{(\ell)}_{F,m}
:=
\operatorname{span}
\left\{
\sym\big(
\boldsymbol n_F^\ell\otimes\boldsymbol t^{\alpha_F}
\big)
:
\alpha_F\in\mathbb T^{d-2}_{m-\ell}
\right\}, \quad 0\le\ell\le m.
$$

\begin{lemma}\label{lem:tn-decomp}
The symmetric tensor space $\mathbb S^{d,m}$ admits the orthogonal $t$-$n$ decomposition
\begin{equation}\label{eq:Sdecomp0}
\mathbb S^{d,m}
=
\Oplus_{\ell=0}^{m}
\mathbb S^{(\ell)}_{F,m}
=
\left(
\Oplus_{\ell=0}^{r-1}
\mathbb S^{(\ell)}_{F,m}
\right)
\Oplus
\sym\big(
\boldsymbol n_F^r\otimes\mathbb S^{d,m-r}
\big),
\qquad 0\le r\le m .
\end{equation}
\end{lemma}

\begin{proof}
The set $\{\boldsymbol t_1,\ldots,\boldsymbol t_{d-1},\boldsymbol n_F\}$ is an orthonormal basis of $\mathbb R^d$. Hence the symmetrized tensor products formed from this basis span $\mathbb S^{d,m}$. The subspace $\mathbb S^{(\ell)}_{F,m}$ consists of tensors with exactly $\ell$ normal factors. Terms with different numbers of normal factors are orthogonal under the Frobenius inner product. Thus the first identity follows.

For the second identity, we use
$$
\Oplus_{\ell=r}^{m}
\mathbb S^{(\ell)}_{F,m}
=
\left\{
\sym(\boldsymbol n_F^r\otimes \eta):
\eta\in
\Oplus_{\ell=0}^{m-r}
\mathbb S^{(\ell)}_{F,m-r}
\right\}
=
\sym\big(\boldsymbol n_F^r\otimes \mathbb S^{d,m-r}\big).
$$
Indeed, after factoring out $r$ normal vectors, by the first identity, $\Oplus_{\ell=0}^{m-r}
\mathbb S^{(\ell)}_{F,m-r}
= \mathbb S^{d,m-r}$. Combining this with the first decomposition proves \eqref{eq:Sdecomp0}.
\end{proof}


\subsection{Geometric decomposition using all faces}

\begin{definition}
For a polytopal element $K\subset\mathbb R^d$, define $\nu_K$ as the largest number of faces of $K$ whose unit outward normal vectors are pairwise non-parallel. Equivalently, $\nu_K$ is the largest integer such that there exist faces $F_1,\ldots,F_{\nu_K}\subset\partial K$ with unit outward normal vectors $\boldsymbol n_{F_1},\ldots,\boldsymbol n_{F_{\nu_K}}$ satisfying
$$
\boldsymbol n_{F_i}
\neq
c\boldsymbol n_{F_j},
\qquad
i\neq j,\quad c\in\mathbb R\setminus\{0\}.
$$
\end{definition}

If $K$ is a bounded full-dimensional polytope, then
$$
d
\le
\nu_K
\le
\#\{F:\ F\subset\partial K\}.
$$
Indeed, if the face normals had fewer than $d$ pairwise non-parallel directions, then they would span a proper subspace of $\mathbb R^d$. A nonzero vector orthogonal to this span would be tangent to all supporting hyperplanes of $K$, contradicting the boundedness of $K$. The upper bound follows from the definition. The lower bound $\nu_K=d$ is achieved by parallelepipeds, while for a $d$-simplex one has $\nu_K=d+1$.

For a fixed $K$, write
$\nu:=\nu_K$. Then
$$
\nu(\left\lfloor\frac{m}{\nu}\right\rfloor+1)>m.
$$
This elementary inequality will be used in the proof below.

\begin{lemma}\label{lem:decomposition_sdm2}
Let $K$ be a polytopal element, and let $F_1,\ldots,F_\nu\subset\partial K$ be faces whose outward unit normal vectors $\boldsymbol n_1,\ldots,\boldsymbol n_\nu$ are pairwise non-parallel. Then
\begin{equation}\label{simplexmeshcondition}
\mathbb S^{d,m}
=
\sum_{i=1}^{\nu}
\sum_{\ell=0}^{\left\lfloor\frac{m}{\nu}\right\rfloor}
\mathbb S^{(\ell)}_{F_i,m}.
\end{equation}
\end{lemma}

\begin{proof}
Let
$
\mathcal W
:=
\sum_{i=1}^{\nu}
\sum_{\ell=0}^{\left\lfloor\frac{m}{\nu}\right\rfloor}
\mathbb S^{(\ell)}_{F_i,m}.
$
It is enough to prove $\mathcal W^\perp=\{0\}$.

Associate to each $\tau\in\mathbb S^{d,m}$ the homogeneous polynomial
$$
p_\tau(x)
:=
\tau:x^{\otimes m}
=
\tau_{i_1,\ldots,i_m}x_{i_1}\cdots x_{i_m},
\qquad x\in\mathbb R^d.
$$
This map is injective.

Suppose $\tau\in\mathcal W^\perp$. Fix $i\in\{1,\ldots,\nu\}$ and decompose
$$
x
=
x_F+x_n\boldsymbol n_i,
\qquad
x_F\perp \boldsymbol n_i,
\quad
x_n=x\cdot\boldsymbol n_i.
$$
Since $\tau$ is orthogonal to $\mathbb S^{(\ell)}_{F_i,m}$ for $\ell=0,\ldots,\left\lfloor\frac{m}{\nu}\right\rfloor$, the polynomial $p_\tau(x_F+x_n\boldsymbol n_i)$ has no terms $x_n^0,\ldots,x_n^{\left\lfloor\frac{m}{\nu}\right\rfloor}$. Hence
$$
p_\tau(x_F+x_n\boldsymbol n_i)
=
x_n^{\left\lfloor\frac{m}{\nu}\right\rfloor+1}q_i(x_F,x_n),
$$
or equivalently,
$$
(x\cdot\boldsymbol n_i)^{\left\lfloor\frac{m}{\nu}\right\rfloor+1}
\mid
p_\tau(x).
$$
This holds for every $i=1,\ldots,\nu$.

Since the linear forms $x\cdot\boldsymbol n_i$, $i=1,\ldots,\nu$, are pairwise non-proportional, they are pairwise coprime. Therefore
$$
\prod_{i=1}^{\nu}
(x\cdot\boldsymbol n_i)^{\left\lfloor\frac{m}{\nu}\right\rfloor+1}
\mid
p_\tau(x).
$$
The degree of this divisor is $\nu(\left\lfloor\frac{m}{\nu}\right\rfloor+1)>m$, whereas $p_\tau$ has degree $m$. Hence $p_\tau\equiv0$. By injectivity of $\tau\mapsto p_\tau$, we obtain $\tau=0$. Thus $\mathcal W^\perp=\{0\}$, and hence $\mathcal W=\mathbb S^{d,m}$.
\end{proof}

\begin{example}[Simplex and parallelepiped]\rm
Let $K$ be a $d$-simplex. Then $\nu=d+1$. If $d\ge m$, then $\lfloor m/(d+1)\rfloor=0$, and 
the tangential tensor spaces from all faces span $\mathbb S^{d,m}$.

If $d<m$, then higher normal layers are needed in general. 
For example, if $d=2$ and $m=5$, then $\nu=3$. Hence
$$
\mathbb S^{2,5}
=
\sum_{F\subset\partial K}
\left(
\mathbb S^{(0)}_{F,5}
+
\mathbb S^{(1)}_{F,5}
\right).
$$

For a parallelepiped in $\mathbb R^d$, opposite faces have parallel normals, so $\nu=d$. Hence the same statement holds with $d+1$ replaced by $d$.
\end{example}

\begin{example}[Polygon with many normal directions]\rm
Let $K$ be a polygon in $\mathbb R^2$ with $N$ edges and no two edge normals parallel. Then $\nu=N$. 
If $N>m$, then $\left\lfloor\frac{m}{N}\right\rfloor=0$. 
Thus, for a polygon with enough edge directions, the tangential layers alone span $\mathbb S^{2,m}$.
\end{example}

\section{Finite Elements for Symmetric Tensors}\label{sec:femsymtensor}

In this section, we construct local finite element spaces for symmetric tensor fields. We first introduce a scalar single-face finite element and its normal Taylor decomposition. We then combine this scalar decomposition with the $t$-$n$ decomposition of symmetric tensors to describe normal trace spaces and bubble spaces. Finally, we assemble these single-face spaces on a polytopal element. 

\subsection{Single-face finite element}

Let $T\subset\mathbb R^d$ be a single-face cell with distinguished face $F$ and vertex $c$ opposite to $F$. Let $\lambda_c$ be the affine function satisfying
$$
\lambda_c(c)=1,
\qquad
\lambda_c|_F=0.
$$
Then $\nabla\lambda_c\perp F$ and
$$
\nabla\lambda_c=(\partial_n\lambda_c)\boldsymbol n_F,
\qquad
\partial_n\lambda_c=\frac{1}{\operatorname{dist}(c,F)}.
$$

Although $T=\operatorname{conv}(c,F)$ need not be a simplex, we may use simplex
coordinates as auxiliary affine coordinates. Choose a nondegenerate
$(d-1)$-simplex $S_F\subset F$ and let
$\widehat T_F:=\operatorname{conv}(c,S_F).$
Let $\lambda_c,\lambda_1,\ldots,\lambda_d$ be the barycentric coordinates of
$\widehat T_F$. These functions are affine on the whole space $\mathbb R^d$.
Therefore their restrictions to $T$ may be used to represent polynomials on
$T$. In particular,
$$
\mathbb P_k(T)
=
\operatorname{span}\{\lambda^\alpha:\alpha\in\mathbb T^d_k\},
$$
where
$$
\lambda^\alpha
=
\lambda_c^{\alpha_0}\lambda_F^{\alpha_F},
\qquad
\lambda_F^{\alpha_F}
=
\lambda_1^{\alpha_1}\cdots\lambda_d^{\alpha_d}.
$$
We identify a polynomial on $F$ with its extension to $T$ obtained by using the
same monomials $\lambda_F^{\alpha_F}$, now viewed as affine functions on $T$.
With this convention, for $k\ge0$ and $i\ge0$, the polynomial space has the
normal-layer decomposition
\begin{equation}\label{eq:Pkgeodecomp1}
\mathbb P_k(T)
=
\left(
\Oplus_{j=0}^{i-1}\lambda_c^j\mathbb P_{k-j}(F)
\right)
\Oplus
\lambda_c^i\mathbb P_{k-i}(T),
\end{equation}
with the convention $\mathbb P_s=\{0\}$ if $s<0$. 

\begin{lemma}\label{lm:Crelement}
Let $T$ be a single-face cell with distinguished face $F$. For integers $k\ge0$ and $m\geq 1$, the space $\mathbb P_k(T)$ is uniquely determined by the following degrees of freedom:
\begin{subequations}\label{eq:PkDoFs}
\begin{align}
\label{eq:PkDoFs1}
\int_F \partial_n^i v\,q\,\mathrm dS,
&\qquad
q\in\mathbb P_{k-i}(F),
\quad i=0,1,\ldots,m-1,\\
\label{eq:PkDoFs2}
\int_T v\,q\,\mathrm dx,
&\qquad
q\in\mathbb P_{k-m}(T).
\end{align}
\end{subequations}
\end{lemma}

\begin{proof}
By \eqref{eq:Pkgeodecomp1} with $i=m$, every $p\in\mathbb P_k(T)$ can be written as
$$
p
=
\sum_{i=0}^{m-1}\lambda_c^i p_i
+
\lambda_c^m p_m,
$$
where $p_i\in\mathbb P_{k-i}(F)$ for $i=0,\ldots,m-1$, and $p_m\in\mathbb P_{k-m}(T)$. Since $\lambda_c|_F=0$ and $\partial_n\lambda_c\ne0$ on $F$, the vanishing of the face DoFs \eqref{eq:PkDoFs1} implies successively that
$$
p_0=p_1=\cdots=p_{m-1}=0.
$$
Hence $p=\lambda_c^m p_m$. Taking $q=p_m$ in \eqref{eq:PkDoFs2} gives
$$
0
=
\int_T \lambda_c^m p_m^2\,\mathrm dx.
$$
Since $\lambda_c\ge0$ on $T$ and is positive in the interior of $T$, we obtain $p_m=0$. Thus $p=0$. Since the number of DoFs equals $\dim\mathbb P_k(T)$, the DoFs are unisolvent.
\end{proof}

Introduce the single-face trace operator
$$
\operatorname{tr}_F^{H^m}v
=
(v,\partial_n v,\ldots,\partial_n^{m-1}v)|_F.
$$
Then \eqref{eq:Pkgeodecomp1} gives
$$
\mathbb B_k^{H^m}(T)
:=
\ker(\operatorname{tr}_F^{H^m})\cap\mathbb P_k(T)
=
\lambda_c^m\mathbb P_{k-m}(T).
$$
When $k<m$, $\mathbb B_k^{H^m}(T)=\{0\}$.

\subsection{Polynomial symmetric tensors and normal traces}

Combining the polynomial decomposition \eqref{eq:Pkgeodecomp1} with the $t$-$n$ decomposition of $\mathbb S^{d,m}$, we obtain
\begin{equation}\label{eq:PkSdecomp}
\mathbb P_k(T;\mathbb S^{d,m})
=
\Oplus_{i=0}^{k}
\Oplus_{\ell=0}^{m}
\lambda_c^i\mathbb P_{k-i}(F;\mathbb S^{(\ell)}_{F,m}).
\end{equation}
This decomposition is indexed by pairs $(i,\ell)$, where $i$ is the polynomial layer in the normal direction and $\ell$ is the number of normal tensor factors; see Fig.~\ref{fig:decompositionmatrix}.

Define the normal trace component
$$
\gamma_s(\tau)
:=
((\operatorname{div}^s\tau)\mathbin{\lrcorner}\boldsymbol n_F)|_F, \qquad 0\le s\le m-1.
$$
The full normal trace is
$$
\operatorname{tr}_F^{\operatorname{div}^m}\tau
:=
(\gamma_0(\tau),\gamma_1(\tau),\ldots,\gamma_{m-1}(\tau)).
$$
For $\tau=a(x)\eta$ with a scalar function $a$ and a constant tensor $\eta\in\mathbb R^{d,m}$, we have
$$
(\operatorname{div}^s\tau)_{i_1\cdots i_{m-s}}
=
\partial_{j_1}\cdots\partial_{j_s}
\big(a\,\eta_{i_1\cdots i_{m-s}j_1\cdots j_s}\big)
=
(\partial_{j_1}\cdots\partial_{j_s}a)
\eta_{i_1\cdots i_{m-s}j_1\cdots j_s}.
$$
Therefore,
$$
\operatorname{div}^s(a\eta)
=
\eta\mathbin{\lrcorner}\nabla^s a.
$$

\begin{lemma}\label{lem:normaltraceblock}
Let $q_F\in\mathbb P_{k-i}(F)$ and $\eta_\ell\in\mathbb S^{(\ell)}_{F,m}$. Then, for $i=0,1,\ldots,m-1$,
\begin{equation}\label{eq:firstnormaltrace}
\gamma_i(\lambda_c^i q_F\eta_\ell)
=
i!(\partial_n\lambda_c)^i q_F
(\eta_\ell\mathbin{\lrcorner}\boldsymbol n_F^{i+1}).
\end{equation}
Moreover, $\gamma_s(\lambda_c^i q_F\eta_\ell)=0$ for $0\le s<i$. Consequently, if $i\ge\ell$, then
$$
\operatorname{tr}_F^{\operatorname{div}^m}(\lambda_c^i q_F\eta_\ell)=0.
$$
If $i<\ell$, then $\gamma_i$ is injective on
$$
\lambda_c^i\mathbb P_{k-i}(F;\mathbb S^{(\ell)}_{F,m}).
$$
In addition, with the convention $\mathbb P_{k-i}(F)=\{0\}$ if $k-i<0$,
\begin{equation}\label{eq:gammai-image}
\gamma_i\left(
\Oplus_{\ell=i+1}^{m}
\lambda_c^i\mathbb P_{k-i}(F;\mathbb S^{(\ell)}_{F,m})
\right)
=
\gamma_i\big(\mathbb P_k(T;\mathbb S^{d,m})\big)
=
\mathbb P_{k-i}(F;\mathbb S^{d,m-i-1}).
\end{equation}
\end{lemma}

\begin{proof}
Figure \ref{fig:decompositionmatrix} illustrates how the trace operator acts on the $(i,\ell)$ block.
Set $p=\lambda_c^i q_F$. Since $\eta_\ell$ is constant,
$$
\operatorname{div}^s(p\eta_\ell)
=
\eta_\ell\mathbin{\lrcorner}\nabla^s p.
$$
Also $\lambda_c|_F=0$ and
$$
\nabla\lambda_c=(\partial_n\lambda_c)\boldsymbol n_F
\qquad \text{on }F.
$$
If $s<i$, every term in $\nabla^s(\lambda_c^i q_F)$ contains a positive power of $\lambda_c$ and hence vanishes on $F$. Thus
$$
\gamma_s(\lambda_c^i q_F\eta_\ell)=0,
\qquad 0\le s<i.
$$

For $s=i$, the only term that survives on $F$ is the term in which all $i$ derivatives hit $\lambda_c^i$. Hence
$$
\nabla^i(\lambda_c^i q_F)|_F
=
i!(\partial_n\lambda_c)^i q_F\boldsymbol n_F^i.
$$
Using
$$
\operatorname{div}^i(\lambda_c^i q_F\eta_\ell)
=
\eta_\ell\mathbin{\lrcorner}\nabla^i(\lambda_c^i q_F),
$$
we get
$$
(\operatorname{div}^i(\lambda_c^i q_F\eta_\ell))|_F
=
i!(\partial_n\lambda_c)^i q_F
(\eta_\ell\mathbin{\lrcorner}\boldsymbol n_F^i).
$$
The trace $\gamma_i$ adds one more contraction with $\boldsymbol n_F$, which proves \eqref{eq:firstnormaltrace}.

If $s>i$, every surviving term still contains the factor produced by differentiating $\lambda_c^i$ in the normal direction. After the additional normal contraction in $\gamma_s$, each such term contains the factor
$$
\eta_\ell\mathbin{\lrcorner}\boldsymbol n_F^{i+1}.
$$
If $i\ge\ell$, this contraction is zero because $\eta_\ell$ contains only $\ell$ normal factors. Therefore all trace components vanish.

If $i<\ell$, then contraction with $\boldsymbol n_F^{i+1}$ maps $\mathbb S^{(\ell)}_{F,m}$ injectively into $\mathbb S^{(\ell-i-1)}_{F,m-i-1}$. Indeed, for a basis tensor
$$
\eta_\ell
=
\sym(\boldsymbol n_F^\ell\otimes\boldsymbol t^{\alpha_F}),
$$
one has
$$
\eta_\ell\mathbin{\lrcorner}\boldsymbol n_F^{i+1}
=
c_{\ell,i,m}\,
\sym(\boldsymbol n_F^{\ell-i-1}\otimes\boldsymbol t^{\alpha_F}),
$$
where $c_{\ell,i,m}\ne0$. Hence \eqref{eq:firstnormaltrace} implies the injectivity of $\gamma_i$ on the stated block.

It remains to prove \eqref{eq:gammai-image}. From \eqref{eq:firstnormaltrace},
$$
\gamma_i\left(
\Oplus_{\ell=i+1}^{m}
\lambda_c^i\mathbb P_{k-i}(F;\mathbb S^{(\ell)}_{F,m})
\right)
=
\mathbb P_{k-i}(F)\otimes
\left(
\Oplus_{\ell=i+1}^{m}
(\mathbb S^{(\ell)}_{F,m}\mathbin{\lrcorner}\boldsymbol n_F^{i+1})
\right).
$$
The contraction by $\boldsymbol n_F^{i+1}$ maps $\mathbb S^{(\ell)}_{F,m}$ onto
$\mathbb S^{(\ell-i-1)}_{F,m-i-1}$ for $\ell=i+1,\ldots,m$. Therefore
$$
\Oplus_{\ell=i+1}^{m}
(\mathbb S^{(\ell)}_{F,m}\mathbin{\lrcorner}\boldsymbol n_F^{i+1})
=
\Oplus_{r=0}^{m-i-1}\mathbb S^{(r)}_{F,m-i-1}
=
\mathbb S^{d,m-i-1}.
$$
Thus the left-hand side of \eqref{eq:gammai-image} equals
$\mathbb P_{k-i}(F;\mathbb S^{d,m-i-1})$.

Conversely, for any $\tau\in\mathbb P_k(T;\mathbb S^{d,m})$, the trace $\gamma_i(\tau)$ contains $i$ derivatives and one normal contraction. Hence
$$
\gamma_i(\tau)\in\mathbb P_{k-i}(F;\mathbb S^{d,m-i-1}).
$$
This gives the reverse inclusion and proves \eqref{eq:gammai-image}.
\end{proof}

Define
$$
\mathbb B_k^{\operatorname{div}^m}(T;\mathbb S^{d,m})
:=
\ker(\operatorname{tr}_F^{\operatorname{div}^m})
\cap
\mathbb P_k(T;\mathbb S^{d,m}).
$$
Lemma~\ref{lem:normaltraceblock} gives the following characterization of the polynomial trace and bubble spaces.

\begin{proposition}\label{prop:bubbledecomp}
For $k\ge0$, the polynomial bubble space is
\begin{equation}\label{eq:BkSdecomp}
\mathbb B_k^{\operatorname{div}^m}(T;\mathbb S^{d,m})
=
\Oplus_{\ell=0}^{m}
\lambda_c^\ell\mathbb P_{k-\ell}(T;\mathbb S^{(\ell)}_{F,m}).
\end{equation}
Moreover, the full trace map satisfies
\begin{equation}\label{eq:traceimagepoly}
\operatorname{tr}_F^{\operatorname{div}^m}
\big(\mathbb P_k(T;\mathbb S^{d,m})\big)
=
\prod_{i=0}^{m-1}
\mathbb P_{k-i}(F;\mathbb S^{d,m-i-1}).
\end{equation}
\end{proposition}

Figure~\ref{fig:decompositionmatrix} shows the $(i,\ell)$-indexed block layout. The lower-triangular cells $i\ge\ell$ form the bubble space, while the upper-triangular cells $i<\ell$ are detected by the normal traces.

\begin{figure}[htp]
\centering
\subfigure[No enrichment for $k\geq m,$ and $\lfloor m/\nu_K\rfloor = 0$]{%
\begin{minipage}[t]{0.33\linewidth}
\centering
\includegraphics[width=3.75cm]{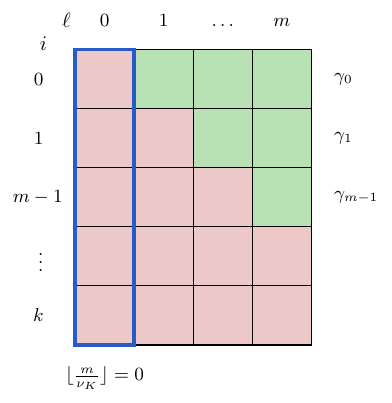}
\end{minipage}%
}%
\subfigure[Trace enrichment]{%
\begin{minipage}[t]{0.33\linewidth}
\centering
\includegraphics[width=3cm]{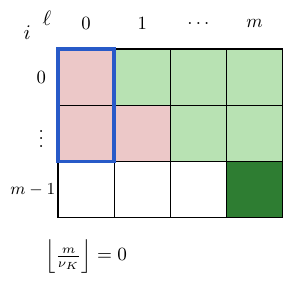}
\end{minipage}%
}%
\subfigure[Trace and bubble enrichment]{%
\begin{minipage}[t]{0.33\linewidth}
\centering
\includegraphics[width=3.75cm]{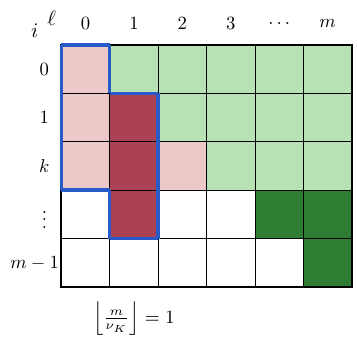}
\end{minipage}%
}
\caption{Block layout of $\mathbb P_k(T)\otimes\mathbb S^{d,m}$ indexed by $(i,\ell)$. Dark green cells are trace-enrichment components, dark red cells are bubble-enrichment components, and the blue-outlined column contains the essential bubble components used in the local stability construction.}
\label{fig:decompositionmatrix}
\end{figure}

\subsection{Single-face enriched tensor space}
The natural local polynomial tensor space is $\mathbb P_k(T;\mathbb S^{d,m})$. To obtain stability for all polynomial degrees, we enrich this space by trace and bubble components. The trace enrichment is needed only when $k<m-1$, while the bubble enrichment depends on the geometry of $K$.

\subsubsection{Trace cells}

When $k\ge m-1$, the polynomial tensor space already contains all trace layers. When $k<m-1$, the trace components $\gamma_i$, $k+1\le i\le m-1$, are missing because $\operatorname{div}^i\sigma=0$ for $\sigma\in\mathbb P_k(T;\mathbb S^{d,m})$; see the white cells in Fig.~\ref{fig:decompositionmatrix}(b).

We will add more trace cells row by row and define the full trace-cell space by
\begin{equation}\label{eq:full-trace-space}
\operatorname{Tr}(\Sigma_k;T)
:=
\Oplus_{i=0}^{m-1}
\lambda_c^i
\mathbb P_{(k-i)^+}(F)\otimes
\sym\big(\boldsymbol n_F^{i+1}\otimes\mathbb S^{d,m-i-1}\big),
\end{equation}
where $(k-i)^+ = \max\{k-i,0\}$.


\subsubsection{Bubble cells}

To prove the inf-sup condition, the local tensor space must contain enough bubble tensors to recover all tensors in $\mathbb P_k(K;\mathbb S^{d,m})$. This requirement depends on the geometry of $K$ through the number $\nu_K$ of pairwise non-parallel face normals.

For fixed $K$, write $\nu:=\nu_K$. If $\lfloor m/\nu\rfloor>0$, we add extra bubble cells so that all tensor layers up to order $\lfloor m/\nu\rfloor$ carry degree $k$. Define
\begin{equation}\label{eq:full-bubble-space}
\begin{aligned}
\mathbb B(\Sigma_k;T)
:={}&
\underbrace{
\Oplus_{\ell=0}^{\lfloor m/\nu\rfloor}
\lambda_c^{\ell}\mathbb P_k(T;\mathbb S^{(\ell)}_{F,m})
}_{\mathbb B_{\mathrm{ess}}(T)}
\Oplus
\underbrace{
\Oplus_{\ell=\lfloor m/\nu\rfloor+1}^{m}
\lambda_c^{\ell}\mathbb P_{k-\ell}(T;\mathbb S^{(\ell)}_{F,m})
}_{\mathbb B_{\mathrm{rem}}(T)} .
\end{aligned}
\end{equation}

\subsubsection{Stress space and unisolvence}

We define the local single-face stress space by the trace--bubble decomposition
\begin{equation}\label{eq:decomposition-tr-bubble}
\Sigma_k(T)
:=
\operatorname{Tr}(\Sigma_k;T)
\Oplus
\mathbb B(\Sigma_k;T).
\end{equation}

The decomposition \eqref{eq:decomposition-tr-bubble} gives the following degrees of freedom.

\begin{lemma}\label{lem:sfTtensor}
Let $m\ge1$ and $k\ge0$. The following DoFs are unisolvent for $\Sigma_k(T)$:
\begin{subequations}\label{sfTtensordof}
\begin{align}
\label{sfTtensordof1}
\int_F \gamma_i(\sigma):\tau\,\mathrm dS,
&\qquad
\tau\in
\mathbb P_{(k-i)^+}(F;\mathbb S^{d,m-i-1}),
\quad i=0,\ldots,m-1,\\
\label{sfTtensordof2}
\int_T \sigma:\tau\,\mathrm dx,
&\qquad
\tau\in \mathbb B(\Sigma_k;T).
\end{align}
\end{subequations}
\end{lemma}

\begin{proof}
Let $\sigma\in\Sigma_k(T)$ and assume that all DoFs in \eqref{sfTtensordof} vanish. Write
$$
\sigma
=
\sigma_{\operatorname{tr}}
+
\sigma_{\operatorname b},
\qquad
\sigma_{\operatorname{tr}}\in\operatorname{Tr}(\Sigma_k;T),
\quad
\sigma_{\operatorname b}\in\mathbb B(\Sigma_k;T).
$$
By the construction of $\operatorname{Tr}(\Sigma_k;T)$ and Lemma~\ref{lem:normaltraceblock}, the face moments \eqref{sfTtensordof1} determine the trace part. Since all face moments vanish, we have $\sigma_{\operatorname{tr}}=0$. Hence $\sigma=\sigma_{\operatorname b}\in\mathbb B(\Sigma_k;T)$. Taking $\tau=\sigma$ in \eqref{sfTtensordof2} gives $\|\sigma\|_{0,T}^2=0.$
Therefore $\sigma=0$. Since the number of DoFs equals $\dim\Sigma_k(T)$, the DoFs are unisolvent.
\end{proof}

\begin{remark}\rm
When $m$ is moderate, one may avoid the minimal enriched tensor space and use the full polynomial tensor space
$$
\mathbb P_{\bar k}(T;\mathbb S^{d,m}),\qquad \bar k= \max\{k+\lfloor m/\nu_K\rfloor, m-1\},\quad T\in\mathcal K_h^{\rm R},
$$
which contains $\operatorname{Tr}(\Sigma_k;T)\Oplus \mathbb B(\Sigma_k;T)$. 
The minimal construction is used in the analysis to identify the needed trace and bubble components, while the full polynomial space gives a simpler implementation.
\end{remark}

\subsection{Finite element for symmetric tensors on a polytope}
We now fix one polytope $K$ and construct the local tensor space on $K$. Let $K^{\rm R}:=\{T_i\}_{i=1}^{n_K}$ be the local refinement associated with the faces $F_i\subset\partial K$. We assume throughout that such a point $c_K$ exists and that each $T_F$ is contained in $K$. Define the broken local stress space by
\begin{equation}\label{eq:broken-local-stress-space}
\Sigma^{-1}_{k}(K^{\rm R})
:=
\prod_{T\in K^{\rm R}}\Sigma_k(T)
=
\operatorname{Tr}(\Sigma_k;K)
\Oplus
\mathbb B_k(\Sigma_k;K),
\end{equation}
where
$$
\operatorname{Tr}(\Sigma_k;K)
:=
\prod_{T_F\in K^{\rm R}}\operatorname{Tr}(\Sigma_k;T_F),
\qquad
\mathbb B_k(\Sigma_k;K)
:=
\prod_{T\in K^{\rm R}}\mathbb B(\Sigma_k;T).
$$
For $\tau\in\mathbb B_k(\Sigma_k;K)$, the exterior traces satisfy
$$
\gamma_i(\tau)|_{\partial K}=0,\qquad i=0,\ldots,m-1.
$$
The trace condition is imposed only on the exterior faces $F_i\subset\partial K$, not on the internal faces of the local refinement.

By Lemma~\ref{lem:sfTtensor} on each refined cell $T$, we obtain the following DoFs on $K$.

\begin{lemma}\label{lem:Ktensor-unisolvence}
Let $K$ be a fixed polytope and let $k\ge0$. The following degrees of freedom are unisolvent for $\Sigma^{-1}_{k}(K^{\rm R})$:
\begin{subequations}\label{Ktensordof}
\begin{align}
\label{Ktensordof1}
\int_F \gamma_i(\sigma):\tau\,\mathrm dS,
&\quad
\tau\in
\mathbb P_{(k-i)^+}(F;\mathbb S^{d,m-i-1}),
\ i=0,\ldots,m-1,\ F\subset\partial K,\\
\label{Ktensordof2}
\int_K \sigma:\tau\,\mathrm dx,
&\quad
\tau\in
\mathbb B_k(\Sigma_k;K).
\end{align}
\end{subequations}
\end{lemma}

To prove the discrete inf-sup condition, we use an equivalent set of interior DoFs. Let
$$
N_m:=\dim\mathbb S^{d,m}=\binom{m+d-1}{m}.
$$
By Lemma~\ref{lem:decomposition_sdm2}, we can choose faces $F_\alpha\subset\partial K$, integers $0\le r_\alpha\le \lfloor m/\nu\rfloor$, and tensors
$$
\{\xi_\alpha\in\mathbb S^{(r_\alpha)}_{F_\alpha,m},
\,
\alpha=1,\ldots,N_m,\} \text{ is a basis of } \mathbb S^{d,m}.
$$
Let $T_\alpha=T_{F_\alpha}$ and $\lambda_\alpha$ be the affine function on $T_\alpha$ satisfying $\lambda_\alpha|_{F_\alpha}=0$ and $\lambda_\alpha(c_K)=1$.

Define the essential bubble space on $K$ as
\begin{equation}\label{eq:essential-bubble-space}
\mathbb B_{\rm ess}(\Sigma_k;K)
:=
\Oplus_{\alpha=1}^{N_m}
\left\{
\lambda_\alpha^{r_\alpha} q_\alpha \xi_\alpha\,\chi_{T_\alpha}:
q_\alpha\in\mathbb P_k(T_\alpha)
\right\}
\subset
\mathbb B_k(\Sigma_k;K),
\end{equation}
where $\chi_{T_\alpha}$ is the characteristic function of $T_\alpha$. 
Let $\mathbb B_{\rm rem}(\Sigma_k;K)$ be the $L^2(K)$-orthogonal complement of $\mathbb B_{\rm ess}(\Sigma_k;K)$ in $\mathbb B_k(\Sigma_k;K)$. 
Then
\begin{equation}\label{eq:bubble-essential-rem-split}
\mathbb B_k(\Sigma_k;K)
=
\mathbb B_{\rm rem}(\Sigma_k;K)
\Oplus^{\bot_{L^2(K)}}
\mathbb B_{\rm ess}(\Sigma_k;K).
\end{equation}
We will change the DoF corresponding to $\mathbb B_{\rm rem}(\Sigma_k;K)$ to $\mathbb P_k(K;\mathbb S^{d,m})$.
\begin{lemma}\label{lem:dof-equal}
Let $K\in\mathcal K_h$ and $k\ge0$. The space $\Sigma_k^{-1}(K^{\rm R})$ is uniquely determined by the following degrees of freedom:
\begin{subequations}\label{Ksymtensordof}
\begin{align}
\label{Ksymtensordof1}
\int_F \gamma_i(\sigma):\tau\,\mathrm dS,
&\quad
\tau\in
\mathbb P_{(k-i)^+}(F;\mathbb S^{d,m-i-1}),
\ i=0,\ldots,m-1,\ F\subset\partial K,\\
\label{Ksymtensordof2}
\int_K \sigma:\tau\,\mathrm dx,
&\quad
\tau\in \mathbb B_{\rm rem}(\Sigma_k;K),\\
\label{Ksymtensordof3}
\int_K \sigma:\tau\,\mathrm dx,
&\quad
\tau\in \mathbb P_k(K;\mathbb S^{d,m}).
\end{align}
\end{subequations}
\end{lemma}

\begin{proof}
By Lemma~\ref{lem:Ktensor-unisolvence}, it is enough to prove that the moments in \eqref{Ksymtensordof2}--\eqref{Ksymtensordof3} are equivalent to the bubble moments in \eqref{Ktensordof2}. Let $\sigma\in\mathbb B_k(\Sigma_k;K)$ and assume that all degrees of freedom \eqref{Ksymtensordof2}--\eqref{Ksymtensordof3} vanish. By \eqref{eq:bubble-essential-rem-split}, write
$$
\sigma=\sigma_{\rm rem}+\sigma_{\rm ess},
\qquad
\sigma_{\rm rem}\in\mathbb B_{\rm rem}(\Sigma_k;K),
\quad
\sigma_{\rm ess}\in\mathbb B_{\rm ess}(\Sigma_k;K).
$$
Due to the $L^2(K)$-orthogonality, taking $\tau=\sigma_{\rm rem}$ in \eqref{Ksymtensordof2} gives
$$
0=(\sigma,\sigma_{\rm rem})_K=\|\sigma_{\rm rem}\|_{0,K}^2,
$$
so $\sigma_{\rm rem}=0$. Hence
\begin{equation}\label{eq:sigma-ess-expansion}
\sigma=\sigma_{\rm ess}
=
\sum_{\alpha=1}^{N_m}
\lambda_\alpha^{r_\alpha}
q_\alpha\xi_\alpha\,\chi_{T_\alpha},
\qquad
q_\alpha\in\mathbb P_k(T_\alpha).
\end{equation}
Let $\{\widehat\xi_\alpha\}_{\alpha=1}^{N_m}$ be the Frobenius-dual basis of $\{\xi_\alpha\}_{\alpha=1}^{N_m}$, so that $\xi_\alpha:\widehat\xi_\beta=\delta_{\alpha\beta}$. Extend each $q_\alpha$ to a polynomial in $\mathbb P_k(K)$, still denoted by $q_\alpha$. Taking $\tau=q_\alpha\widehat\xi_\alpha$ in \eqref{Ksymtensordof3} gives
$$
0
=
(\sigma,q_\alpha\widehat\xi_\alpha)_K
=
\int_{T_\alpha}
\lambda_\alpha^{r_\alpha}q_\alpha^2\,\mathrm dx .
$$
Since $\lambda_\alpha^{r_\alpha}\ge0$ on $T_\alpha$ and is positive in its interior, $q_\alpha=0$. Thus $\sigma=0$. The number of degrees of freedom is unchanged by the splitting \eqref{eq:bubble-essential-rem-split} and $\dim\mathbb B_{\rm ess}(\Sigma_k;K)=\dim\mathbb P_k(K;\mathbb S^{d,m})$. Hence the degrees of freedom are unisolvent.
\end{proof}

\section{Staggered Discontinuous Galerkin Methods}\label{sec:sdg}

In this section, we define the staggered discontinuous Galerkin method for the polyharmonic problem. The tensor variable is normal-trace continuous across primal faces, while the scalar variable is locally $H^m$ on each primal element. 

\subsection{Primal and dual meshes}\label{sec:mesh}

Let $\mathcal K_h$ be a shape-regular polytopal partition of $\Omega\subset\mathbb R^d$. Let $\mathcal F_h$ be the set of faces, $\mathcal F_h^\partial$ the set of boundary faces, and
$\mathring{\mathcal F}_h:=\mathcal F_h\setminus\mathcal F_h^\partial$
the set of interior faces. We call $\mathcal K_h$ the primal mesh.

For each $K\in\mathcal K_h$, choose a point $c_K\in K$. For each face $F\subset\partial K$, define
$$
T_F:=\operatorname{conv}(c_K,F),
\qquad
K^{\rm R}:=\{T_F:\ F\subset\partial K\}.
$$
The collection of all such cells gives a refinement $\mathcal K_h^{\rm R}$ of $\mathcal K_h$. We assume that the points $c_K$ are chosen so that $\mathcal K_h^{\rm R}$ is uniformly shape regular, and that the number of faces of each element is uniformly bounded. 

Let $\mathcal F_h^{\rm R}$ be the set of faces of the refined mesh $\mathcal K_h^{\rm R}$ and define $\mathring{\mathcal F}_h^* := \mathcal F_h^{\rm R}\setminus\mathcal F_h.$ The faces in $\mathring{\mathcal F}_h^*$ are the dual faces; they lie inside primal elements. For each primal face $F\in\mathcal F_h$, let $\omega_F$ be the union of refined cells in $\mathcal K_h^{\rm R}$ that contain $F$. The dual mesh is
$$
\mathcal K_h^*
:=
\{\omega_F:\ F\in\mathcal F_h\}.
$$

For each face $F$, fix a unit normal $\boldsymbol n_F$. When no confusion can arise, we write $\boldsymbol n$ for the relevant unit normal. For an interior face $F$, the jump $[\cdot]$ is taken with respect to $\boldsymbol n_F$. On a boundary face, we use the convention $[w]=w$.

\begin{figure}[htp]
\begin{center}
\includegraphics[width=4.5cm]{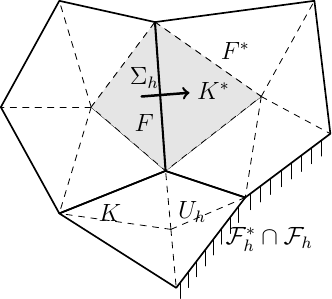}
\caption{Staggered meshes. Solid lines represent primal faces, while dashed lines represent dual faces.}
\label{fig:SDGmesh}
\end{center}
\end{figure}

\subsection{Broken spaces and weak operators}

Define the broken scalar space
$$
H^m(\mathcal K_h)
:=
\{v\in L^2(\Omega):\ v|_K\in H^m(K)\quad\text{for all }K\in\mathcal K_h\}.
$$
For the tensor variable, define
$$
H(\operatorname{div}^m,\mathcal K_h^*;\mathbb S^{d,m})
:=
\left\{
\tau\in L^2(\Omega;\mathbb S^{d,m}):
\tau|_{K^*}\in H(\operatorname{div}^m,K^*;\mathbb S^{d,m})
\ \text{for all }K^*\in\mathcal K_h^*
\right\}.
$$
Equivalently,
$$
H(\operatorname{div}^m,\mathcal K_h^*;\mathbb S^{d,m})
=
\left\{
\tau\in H^m(\mathcal K_h^{\rm R};\mathbb S^{d,m}):
[\operatorname{tr}_F^{\operatorname{div}^m}\tau]=0
\quad\text{for all }F\in\mathring{\mathcal F}_h
\right\},
$$
where
$$
\operatorname{tr}_F^{\operatorname{div}^m}\tau =
(\gamma_0(\tau),\gamma_1(\tau),\ldots,\gamma_{m-1}(\tau)),
\
\gamma_i(\tau) =
((\operatorname{div}^i\tau)\mathbin{\lrcorner}\boldsymbol n_F)|_F,
\
i = 0, \ldots, m-1.
$$
Thus the normal trace components are single-valued across primal faces. The tensor variable may jump across dual faces, while the scalar variable may jump across primal faces.

For $u\in H^m(\mathcal K_h)$, define the weak gradient by
\begin{equation}\label{eq:weakgrad}
\begin{aligned}
\langle \nabla_w^m u,\tau\rangle
:=\;&
\sum_{K\in\mathcal K_h}
(\nabla^m u,\tau)_K\\
&+
\sum_{j=0}^{m-1}
(-1)^{m-j}
\sum_{F\in\mathcal F_h}
\left(
[\nabla^j u],
(\operatorname{div}^{m-1-j}\tau)\mathbin{\lrcorner}\boldsymbol n
\right)_F,
\end{aligned}
\end{equation}
for all $\tau\in H(\operatorname{div}^m,\mathcal K_h^*;\mathbb S^{d,m})$. The jump terms in \eqref{eq:weakgrad} appear only on primal faces, because $u$ is single-valued and $H^m$ inside each primal element.

For $\sigma\in H(\operatorname{div}^m,\mathcal K_h^*;\mathbb S^{d,m})$, define the weak divergence by
\begin{equation}\label{eq:weakdiv}
\begin{aligned}
\langle \operatorname{div}_w^m\sigma,v\rangle
:=\;&
\sum_{K^*\in\mathcal K_h^*}
(\operatorname{div}^m\sigma,v)_{K^*}\\
&+
\sum_{i=0}^{m-1}
(-1)^{m-i}
\sum_{F^*\in\mathring{\mathcal F}_h^*}
\left(
[(\operatorname{div}^{i}\sigma)\mathbin{\lrcorner}\boldsymbol n],
\nabla^{m-i-1}v
\right)_{F^*},
\end{aligned}
\end{equation}
for all $v\in H^m(\mathcal K_h)$. The jump terms in \eqref{eq:weakdiv} appear only on dual faces, because the corresponding normal traces of $\sigma$ are single-valued on primal faces.

\begin{lemma}[Adjointness]\label{lem:weak-adjointness}
For all $\sigma\in H(\operatorname{div}^m,\mathcal K_h^*;\mathbb S^{d,m})$ and all $v\in H^m(\mathcal K_h)$,
\begin{equation}\label{eq:weak-adjointness}
\langle \operatorname{div}_w^m\sigma,v\rangle
=
(-1)^m
\langle \sigma, \nabla_w^m v\rangle .
\end{equation}
\end{lemma}

\begin{proof}
Integrating by parts $m$ times on each refined cell $T\in\mathcal K_h^{\rm R}$ gives volume terms together with face terms. On primal faces, the normal trace components
$$
(\operatorname{div}^{i}\sigma)\mathbin{\lrcorner}\boldsymbol n,
\qquad i=0,\ldots,m-1,
$$
are single-valued by the definition of $H(\operatorname{div}^m,\mathcal K_h^*;\mathbb S^{d,m})$. On dual faces, the traces of $\nabla^{m-i-1}v$ are single-valued because $v\in H^m(K)$ on each primal element. Collecting the remaining primal- and dual-face terms gives exactly \eqref{eq:weakdiv} and \eqref{eq:weakgrad}, and hence \eqref{eq:weak-adjointness}.
\end{proof}


For $k\ge0$, define the scalar space
$$
U_{h,k+m}
:=
\{u_h\in L^2(\Omega):\ u_h|_K\in\mathbb P_{k+m}(K)
\quad\text{for all }K\in\mathcal K_h\}.
$$
Let $\Sigma_k(T)$ be the local tensor space \eqref{eq:decomposition-tr-bubble} on each refined cell $T\in K^{\rm R}$. Define
$$
\Sigma_k^{-1}(K^{\rm R})
:=
\prod_{T\in K^{\rm R}}\Sigma_k(T).
$$
The global tensor space is
\begin{align}
\Sigma_{h,k}
:=
\big\{
\sigma_h\in L^2(\Omega;\mathbb S^{d,m}):
&\ \sigma_h|_K\in\Sigma_k^{-1}(K^{\rm R})
\quad\text{for all }K\in\mathcal K_h,
\notag\\
&\ \text{the trace DoFs in \eqref{Ktensordof1} are single-valued across every }
F\in\mathring{\mathcal F}_h
\big\}.
\label{eq:global-stress-space}
\end{align}
Then
$$
\Sigma_{h,k}
\subset
H(\operatorname{div}^m,\mathcal K_h^*;\mathbb S^{d,m}).
$$


\subsection{The Staggered Discontinuous Galerkin method}

For $\sigma_h,\tau_h\in\Sigma_{h,k}$ and $v_h\in U_{h,k+m}$, define
\begin{align*}
a(\sigma_h,\tau_h)
&:=
(\sigma_h,\tau_h), \\
b_h(\sigma_h,v_h)
&:=
(-1)^{m-1}
\langle \operatorname{div}_w^m\sigma_h,v_h\rangle
=
-\langle \nabla_w^m v_h,\sigma_h\rangle.
\end{align*}
Equivalently, by \eqref{eq:weakgrad},
\begin{equation}\label{eq:bh-expanded}
\begin{aligned}
b_h(\sigma_h,v_h)
={}&
-\sum_{K\in\mathcal K_h}
(\sigma_h,\nabla^m v_h)_K \\
&+
\sum_{i=0}^{m-1}
(-1)^i
\sum_{F\in\mathcal F_h}
\left(
(\operatorname{div}^i\sigma_h)\mathbin{\lrcorner}\boldsymbol n,
[\nabla^{m-i-1}v_h]
\right)_F .
\end{aligned}
\end{equation}

We reindex $i$ by $j$ with relation
\begin{equation}\label{eq:kj-def}
j=m-i-1, k_j:=\max\{k-(m-1-j),0\}=(k-i)^+.
\end{equation}
This index $k_j$ is the effective polynomial degree of the stress trace paired with
$[\nabla^j v_h]$ in \eqref{eq:bh-expanded}. Indeed, with $j=m-i-1$, the jump
$[\nabla^j v_h]$ is paired with
$$
\gamma_{m-1-j}(\sigma_h)
=
(\operatorname{div}^{m-1-j}\sigma_h)\mathbin{\lrcorner}\boldsymbol n,
$$
whose degree is at most $k_j$. Hence only the projected jump
$Q_{F,k_j}[\nabla^j v_h]$ is seen by $b_h(\cdot,\cdot)$. Here $Q_{F,r}$ is the $L^2(F)$ projection onto $\mathbb P_r(F)$, applied componentwise.

For $v\in H^m(\mathcal K_h)$, define
\begin{equation}\label{eq:U-norm}
\|v\|_{m,h}^2
:=
|v|_{H^m(\mathcal K_h)}^2
+
\sum_{j=0}^{m-1}
\sum_{F\in\mathcal F_h}
h_F^{2j-2m+1}
\left\|
Q_{F,k_j}
[\nabla^j v]
\right\|_F^2.
\end{equation}
By the Poincar\'e inequality for broken $H^1$ spaces in~\cite[Remark 1.1]{Brenner2003},
\begin{equation}\label{eq:hybrid-poincare}
\|v\|_{H^{m-1}(\mathcal K_h)}
\lesssim
\|v\|_{m,h}
\qquad
\forall~v\in H^m(\mathcal K_h).
\end{equation}
Thus $\|\cdot\|_{m,h}$ is a norm on $H^m(\mathcal K_h)$.

For $\tau_h\in\Sigma_{h,k}$, define
\begin{equation}\label{eq:Sigma-norm}
\|\tau_h\|_{0,h}^2
:=
\|\tau_h\|^2
+
\sum_{i=0}^{m-1}
\sum_{F\in\mathcal F_h}
h_F^{2i+1}
\left\|
(\operatorname{div}^i\tau_h)\mathbin{\lrcorner}\boldsymbol n
\right\|_F^2.
\end{equation}
By inverse trace inequalities on the finite-dimensional local tensor spaces,
\begin{equation}\label{eq:sigmaJ-equivalence}
\|\tau_h\|_{0,h}
\eqsim
\|\tau_h\|
\qquad
\forall~\tau_h\in\Sigma_{h,k}.
\end{equation}

The SDG method is: find $(\sigma_h,u_h)\in\Sigma_{h,k}\times U_{h,k+m}$ such that
\begin{equation}\label{SDG}
\begin{cases}
a(\sigma_h,\tau_h)+b_h(\tau_h,u_h)=0,
& \quad \forall~\tau_h\in\Sigma_{h,k},\\
b_h(\sigma_h,v_h)=-(f,v_h),
& \quad \forall~v_h\in U_{h,k+m}.
\end{cases}
\end{equation}

By \eqref{eq:sigmaJ-equivalence}, $a(\cdot,\cdot)$ is coercive on $\Sigma_{h,k}$ with respect to $\|\cdot\|_{0, h}$:
\begin{equation}\label{eq:acoercive}
a(\tau_h,\tau_h)
=
\|\tau_h\|^2
\gtrsim
\|\tau_h\|_{0, h}^2.
\end{equation}
Moreover, based on \eqref{eq:bh-expanded}, using the Cauchy-Schwarz inequality with correct scaling, we have
\begin{equation}\label{eq:bcontinuity}
|b_h(\sigma_h,v)|
\lesssim
\|\sigma_h\|_{0, h}\|v\|_{m, h}
\qquad
\forall~\sigma_h\in\Sigma_{h,k},\quad v\in H^m(\mathcal K_h).
\end{equation}

The key is to verify the following inf-sup condition.
\begin{lemma}[Discrete inf-sup condition]\label{lem:discreteinfsup}
There exists a constant $\beta>0$, independent of $h$, such that
\begin{equation}\label{discreteinfsup}
\inf_{u_h\in U_{h,k+m}}
\sup_{0\ne \sigma_h\in\Sigma_{h,k}}
\frac{b_h(\sigma_h,u_h)}
{\|\sigma_h\|_{0, h}\|u_h\|_{m, h}}
\ge \beta .
\end{equation}
\end{lemma}
\begin{proof}
Let $u_h\in U_{h,k+m}$ be fixed. Since
$u_h|_K\in\mathbb P_{k+m}(K)$, we have
$\nabla^m u_h|_K\in \mathbb P_k(K;\mathbb S^{d,m}).$
We define $\sigma_h\in\Sigma_{h,k}$ by the degrees of freedom in
Lemma~\ref{lem:dof-equal}.

First, for each $K\in\mathcal K_h$, prescribe the polynomial moment DoFs by
\begin{equation}\label{eq:infsup-volume-dof}
(\sigma_h,\tau)_K
=
-(\nabla^m u_h,\tau)_K
\qquad
\forall~\tau\in\mathbb P_k(K;\mathbb S^{d,m}).
\end{equation}
Second, on each face $F\in\mathcal F_h$ and for $i=0,\ldots,m-1$, prescribe the trace DoFs by
\begin{equation}\label{eq:infsup-face-dof}
(\operatorname{div}^i\sigma_h)\mathbin{\lrcorner}\boldsymbol n
=
(-1)^i h_F^{-2i-1}
Q_{F,(k-i)^+}
[\nabla^{m-i-1}u_h]
\qquad\text{on }F,
\end{equation}
in the sense of the face moments \eqref{Ksymtensordof1}. On an interior face, the jump is taken with respect to the fixed normal $\boldsymbol n_F$, so the prescribed trace is single-valued. Finally, set all remaining bubble DoFs in
$\mathbb B_{\rm rem}(\Sigma_k;K)$ to zero. By Lemma~\ref{lem:dof-equal}, these data determine a unique
$\sigma_h\in\Sigma_{h,k}$.

Using \eqref{eq:bh-expanded}, \eqref{eq:infsup-volume-dof}, and
\eqref{eq:infsup-face-dof}, we obtain
\begin{align*}
b_h(\sigma_h,u_h)
&=
|u_h|_{H^m(\mathcal K_h)}^2 +
\sum_{i=0}^{m-1}
\sum_{F\in\mathcal F_h}
h_F^{-2i-1}
\left\|
Q_{F,(k-i)^+}
[\nabla^{m-i-1}u_h]
\right\|_F^2 .
\end{align*}
Re-indexing the jump terms by $j=m-i-1$ and using $(k-i)^+ = k_j$ gives
\[
b_h(\sigma_h,u_h)
=
\|u_h\|_{m, h}^2 .
\]

It remains to bound $\sigma_h$. By the scaling of the degrees of freedom in
Lemma~\ref{lem:dof-equal}, the prescribed face moments give
\[
\sum_{i=0}^{m-1}
\sum_{F\in\mathcal F_h}
h_F^{2i+1}
\left\|
(\operatorname{div}^i\sigma_h)\mathbin{\lrcorner}\boldsymbol n
\right\|_F^2
\lesssim
\sum_{i=0}^{m-1}
\sum_{F\in\mathcal F_h}
h_F^{-2i-1}
\left\|
Q_{F,(k-i)^+}
[\nabla^{m-i-1}u_h]
\right\|_F^2 .
\]
Similarly, the polynomial moment DoFs \eqref{eq:infsup-volume-dof}, together with the fact that the
$\mathbb B_{\rm rem}(\Sigma_k;K)$ DoFs are zero, give
\[
\|\sigma_h\|^2
\lesssim
|u_h|_{H^m(\mathcal K_h)}^2
+
\sum_{i=0}^{m-1}
\sum_{F\in\mathcal F_h}
h_F^{-2i-1}
\left\|
Q_{F,(k-i)^+}
[\nabla^{m-i-1}u_h]
\right\|_F^2 .
\]
Hence
\[
\|\sigma_h\|_{0, h}
\lesssim
\|u_h\|_{m, h}.
\]
Therefore
\[
\sup_{0\ne\sigma_h\in\Sigma_{h,k}}
\frac{b_h(\sigma_h,u_h)}{\|\sigma_h\|_{0, h}}
\gtrsim
\|u_h\|_{m, h}.
\]
Taking the infimum over $u_h\in U_{h,k+m}$ proves
\eqref{discreteinfsup}.
\end{proof}

Then by the Babu\v{s}ka--Brezzi theory, we have the following well-posedness.
\begin{theorem}\label{thm:SDG-wellposed}
The SDG method \eqref{SDG} is well posed. Moreover,
\begin{equation}\label{eq:SDG-stability}
\|\sigma_h\|_{0, h}
+
\|u_h\|_{m, h}
\lesssim
\sup_{0\ne v_h\in U_{h,k+m}}
\frac{(f,v_h)}{\|v_h\|_{m, h}}.
\end{equation}
\end{theorem}
%
%

\section{Hybridization}\label{sec:hybrid}

This section hybridizes the staggered DG method by using a fully broken stress space and scalar trace unknowns only on primal faces. The resulting mixed method is stable, and local elimination of the stress variable gives an equivalent stabilization-free weak Galerkin formulation. An optimal-order error analysis is given in the energy norm. 

\subsection{Broken spaces and weak operators}

Following the standard hybridization idea~\cite{arnold1985mixed}, we use a fully broken stress space on the refined mesh $\mathcal K_h^{\rm R}$ and scalar trace unknowns on primal faces. For $j=0,\ldots,m-1$, recall that
$k_j=\max\{k-m+1+j,0\}$ is the degree of the stress trace paired with the $j$-th scalar trace. Define the hybrid scalar space by
\begin{align*}
M_h
:=
\big\{
u_h={}&(u_0,u_{g,0},\ldots,u_{g,m-1}):\
u_0|_K\in\mathbb P_{k+m}(K)
\quad \text{for all }K\in\mathcal K_h,\\
&\quad
u_{g,j}|_F
\in
\mathbb P_{k_j}(F;\mathbb S^{d,j})
\quad \text{for all }F\in\mathcal F_h,\
j=0,\ldots,m-1
\big\}.
\end{align*}
Thus only the traces on primal faces are independent hybrid unknowns. Let
$$
M_h^0
:=
\{u_h\in M_h:\ u_{g,j}|_{\partial\Omega}=0,\quad j=0,\ldots,m-1\}.
$$

The stress space is fully broken on the refined mesh:
\begin{equation}\label{eq:broken-hybrid-tensor-space}
\Sigma_{h,k}^{-1}
:=
\prod_{T\in\mathcal K_h^{\rm R}}\Sigma_k(T).
\end{equation}

For $u_h\in M_h$, define $\nabla_w^m u_h\in\Sigma_{h,k}^{-1}$ cellwise. Let $T=T_F\subset K$ be the refined cell attached to the primal face $F\subset\partial K$. Then $\nabla_{w,T}^m u_h\in\Sigma_k(T)$ is defined by
\begin{equation}\label{eq:hybrid-weak-gradient}
\begin{aligned}
(\nabla_{w,T}^m u_h,\tau)_T
={}&
(\nabla^m u_0,\tau)_T\\
&+
\sum_{j=0}^{m-1}
(-1)^{m-1-j}
\langle
u_{g,j}-\nabla^j u_0,
\gamma_{m-1-j}(\tau)
\rangle_F ,
\end{aligned}
\end{equation}
for all $\tau\in\Sigma_k(T)$. 
No independent trace unknowns are introduced on dual faces; equivalently, the traces on dual faces are taken to be the natural traces $\nabla^j u_0$. Notice that as $\gamma_{m-1-j}(\tau)\in \mathbb P_{k_j}(F;\mathbb S^{d,j})$, we can add the projection operator:
$$\langle
u_{g,j}-\nabla^j u_0,
\gamma_{m-1-j}(\tau)
\rangle_F = \langle
Q_{F, k_j}(u_{g,j}-\nabla^j u_0),
\gamma_{m-1-j}(\tau)
\rangle_F.
$$

\subsection{Hybridized variational form}

Define
$$
a(\sigma_h,\tau_h):=(\sigma_h,\tau_h),
\qquad
b_h(\tau_h,u_h):=-(\tau_h,\nabla_w^m u_h).
$$
The hybridized SDG method is: find
$(\sigma_h,u_h)\in\Sigma_{h,k}^{-1}\times M_h^0$ such that
\begin{subequations}\label{HSDG}
\begin{align}
a(\sigma_h,\tau_h)+b_h(\tau_h,u_h)&=0,
\quad &&\forall~\tau_h\in\Sigma_{h,k}^{-1},\label{ah}\\
b_h(\sigma_h,v_h)&=-(f,v_0),
\quad &&\forall~v_h\in M_h^0.\label{bh}
\end{align}
\end{subequations}
Since the stress space $\Sigma_{h,k}^{-1}$ is fully broken on
$\mathcal K_h^{\rm R}$, the stress variable can be eliminated cell by cell.
Indeed, \eqref{ah} gives
$$
(\sigma_h,\tau_h)
=
(\tau_h,\nabla_w^m u_h)
\qquad
\forall~\tau_h\in\Sigma_{h,k}^{-1},
$$
and hence $\sigma_h=\nabla_w^m u_h.$
Substituting this identity into \eqref{bh} gives the primal hybridized form:
find $u_h\in M_h^0$ such that
\begin{equation}\label{eq:WG}
(\nabla_w^m u_h,\nabla_w^m v_h)
=
(f,v_0)
\qquad
\forall~v_h\in M_h^0 .
\end{equation}
Thus the hybridized SDG method is equivalent to a stabilization-free weak
Galerkin method. Since $\sigma_h$ is eliminated locally, the local stress space
may be enriched to $\mathbb P_{\bar k}(T;\mathbb S^{d,m})$ without adding
global unknowns.

For $u_h\in M_h$, define
\begin{equation}\label{eq:hybrid-u-norm}
\begin{aligned}
\|u_h\|_{m,h}^2
:=
|u_0|_{H^m(\mathcal K_h)}^2
+
\sum_{j=0}^{m-1}
\sum_{K\in\mathcal K_h}
\sum_{F\subset\partial K}
h_F^{2j-2m+1}
\|Q_{F,k_j}(u_{g,j}-\nabla^j u_0)\|_F^2 .
\end{aligned}
\end{equation}
\begin{lemma}\label{lem:hybrid-u-norm-is-norm}
The quantity $\|\cdot\|_{m,h}$ defined by \eqref{eq:hybrid-u-norm} is a norm on $M_h^0$.
\end{lemma}

\begin{proof}
It remains to verify definiteness. Let $u_h=(u_0,u_g)\in M_h^0$ and assume $\|u_h\|_{m,h}=0$. Then $|u_0|_{H^m(\mathcal K_h)}=0$, so $u_0|_T\in P_{m-1}(T)$ on each single-face cell $T$.

We first show $u_0=0$ on every boundary single-face cell. Let $T$ be such a cell and let $F\subset \partial T\cap\partial\Omega$ be its boundary face. Since $u_h\in M_h^0$, the boundary trace variables vanish on $F$. Hence, for each $j=0,\dots,m-1$, $Q_{F,k_j}((\nabla^j u_0):n^{\otimes j})=0$. Because $Q_{F,k_j}$ is the $L^2(F)$-projection onto $\mathbb P_{k_j}(F)$ and $k_j\ge 0$, this implies $\int_F \partial_n^j u_0\,dS=0$ for $j=0,\dots,m-1$. Applying Lemma~\ref{lm:Crelement} with $k=m-1$, we obtain $u_0|_T=0$. Since $u_0\in \mathbb P_{k+m}(K)$ on each primal element $K$, it follows that $u_0|_K=0$ whenever $T\subset K$ is a boundary single-face cell.

Now let $K'$ be a neighboring primal element sharing a primary face $F$ with such an element $K$. Since $\|u_h\|_{m,h}=0$ and $u_0|_K=0$, we get $u_{g,j}|_F=0$ for $j=0,\dots,m-1$. Repeating the same argument on $K'$, we deduce that $u_0|_{K'}=0$. Since the mesh is connected, this propagation argument shows that $u_0=0$ on $\mathcal K_h$.

Finally, returning to \eqref{eq:hybrid-u-norm}, we have $Q_{F,k_j}(u_{g,j})=0$ for all $F$ and $j=0,\dots,m-1$. Since $u_{g,j}\in \mathbb P_{k_j}(F)$, it follows that $u_{g,j}=0$. Hence $u_h=0$, and therefore $\|\cdot\|_{m,h}$ is a norm on $M_h^0$.
\end{proof}

\begin{proposition}[Coercivity of the weak gradient]\label{prop:weak-gradient-coercivity}
There exists a constant $c>0$, independent of $h$, such that
\begin{equation}\label{eq:weak-gradient-coercivity}
\|\nabla_w^m v_h\|^2
\ge
c\|v_h\|_{m,h}^2,
\qquad
\forall~v_h\in M_h^0.
\end{equation}
Consequently, the weak Galerkin formulation \eqref{eq:WG} admits a unique solution.
\end{proposition}

\begin{proof}
Let $v_h\in M_h^0$ be fixed. We construct
$\sigma_h\in\Sigma_{h,k}^{-1}$ locally by the degrees of freedom in
Lemma~\ref{lem:dof-equal}. On each primal element $K$, prescribe the volume
moments by
\begin{equation}\label{eq:WG-coercivity-volume-dof}
(\sigma_h,\tau)_K
=
-(\nabla^m v_0,\tau)_K
\qquad
\forall~\tau\in\mathbb P_k(K;\mathbb S^{d,m}).
\end{equation}
On each primal face $F\subset\partial K$ and for $j=0,\ldots,m-1$, prescribe the
trace moments by
\begin{equation}\label{eq:WG-coercivity-face-dof}
\gamma_{m-1-j}(\sigma_h)
=
(-1)^{m-j}
h_F^{2j-2m+1}
Q_{F,k_j}(v_{g,j}-\nabla^j v_0)
\qquad\text{on }F,
\end{equation}
in the sense of the corresponding trace degrees of freedom. Set the remaining
bubble degrees of freedom to zero. By Lemma~\ref{lem:dof-equal}, these data
determine a unique $\sigma_h\in\Sigma_{h,k}^{-1}$.

Using the definition of $\nabla_w^m$ and the choices
\eqref{eq:WG-coercivity-volume-dof}--\eqref{eq:WG-coercivity-face-dof}, we obtain
\begin{align*}
-(\sigma_h,\nabla_w^m v_h)
={}&
|v_0|_{H^m(\mathcal K_h)}^2+
\sum_{j=0}^{m-1}
\sum_{K\in\mathcal K_h}
\sum_{F\subset\partial K}
h_F^{2j-2m+1}
\left\|
Q_{F,k_j}(v_{g,j}-\nabla^j v_0)
\right\|_F^2\\
={}&
\|v_h\|_{m,h}^2 .
\end{align*}
Moreover, the scaling of the same degrees of freedom gives
\begin{equation}\label{eq:WG-coercivity-sigma-L2-bound}
\|\sigma_h\|
\lesssim
\|v_h\|_{m,h}.
\end{equation}
Therefore
$$
\|v_h\|_{m,h}^2
=
-(\sigma_h,\nabla_w^m v_h)
\le
\|\sigma_h\|\,\|\nabla_w^m v_h\|
\lesssim
\|v_h\|_{m,h}\|\nabla_w^m v_h\|.
$$
If $v_h\ne0$, division by $\|v_h\|_{m,h}$ gives
$$
\|v_h\|_{m,h}
\lesssim
\|\nabla_w^m v_h\|.
$$
The case $v_h=0$ is trivial. This proves \eqref{eq:weak-gradient-coercivity}.

The bilinear form in \eqref{eq:WG} is symmetric and positive definite on
$M_h^0$ by \eqref{eq:weak-gradient-coercivity}. Hence \eqref{eq:WG} admits a
unique solution.
\end{proof}

\subsection{Error analysis}

All constants below may depend on $m$, $d$, $k$, and shape-regularity parameters, but not on $h$. Let
$$
Q_Mu:=(Q_0u,Q_{g,0}u,\ldots,Q_{g,m-1}u)\in M_h^0
$$
be the componentwise projection. Here $Q_0u|_K$ is the $L^2(K)$ projection onto $\mathbb P_{k+m}(K)$, and $Q_{g,j}u|_F$ is the $L^2(F)$ projection of $\nabla^j u$ onto $\mathbb P_{k_j}(F;\mathbb S^{d,j})$ on primal faces.

\begin{lemma}[Weak-gradient approximation]\label{lem:weak-gradient-approx}
Let $Q_Mu\in M_h^0$ be defined as above. Assume
$u\in H^{k+m+1}(\Omega)$. Then
\begin{equation}\label{eq:weak-gradient-approx}
\|\nabla_w^m Q_Mu-\nabla^m u\|
\lesssim
h^{k+1}\|u\|_{k+m+1}.
\end{equation}
\end{lemma}

\begin{proof}
Let $T=T_F\subset K$ and $\tau\in\Sigma_k(T)$. By the definition of
$\nabla_w^m$,
\begin{align*}
(\nabla_w^m Q_Mu-\nabla^m u,\tau)_T
={}&
(\nabla^m Q_0u-\nabla^m u,\tau)_T\\
&+
\sum_{j=0}^{m-1}
(-1)^{m-1-j}
\langle
Q_{g,j}u-\nabla^j Q_0u,
\gamma_{m-1-j}(\tau)
\rangle_F .
\end{align*}
Since
$\gamma_{m-1-j}(\tau)$ belongs to the trace space
$\mathbb P_{k_j}(F;\mathbb S^{d,j})$ and $Q_{g,j}$ is the $L^2(F)$ projection onto this space, we have
$$
\langle
Q_{g,j}u-\nabla^j u,
\gamma_{m-1-j}(\tau)
\rangle_F
=0.
$$
Thus we can replace $Q_{g,j}u$ by $\nabla^j u$ in the formula and obtain
\begin{align*}
(\nabla_w^m Q_Mu-\nabla^m u,\tau)_T
={}&
(\nabla^m(Q_0u-u),\tau)_T\\
&+
\sum_{j=0}^{m-1}
(-1)^{m-1-j}
\langle
\nabla^j(u-Q_0u),
\gamma_{m-1-j}(\tau)
\rangle_F .
\end{align*}

By the standard approximation estimate for the $L^2$ projection
$Q_0$ onto $\mathbb P_{k+m}(K)$,
$$
|u-Q_0u|_{H^s(K)}
\lesssim
h_K^{k+m+1-s}|u|_{H^{k+m+1}(K)},
\qquad 0\le s\le k+m+1.
$$
Hence
$$
\|\nabla^m(Q_0u-u)\|_T
\lesssim
h_K^{k+1}|u|_{H^{k+m+1}(K)}.
$$
For the face term, the trace inequality gives
$$
\|\nabla^j(u-Q_0u)\|_F
\lesssim
h_K^{k+m+1-j-1/2}|u|_{H^{k+m+1}(K)}.
$$
On the other hand, the inverse trace estimate on $\Sigma_k(T)$ gives
$$
\|\gamma_{m-1-j}(\tau)\|_F
\lesssim
h_T^{-(m-1-j)-1/2}\|\tau\|_T .
$$
Combining the last two estimates yields
$$
\left|
\langle
\nabla^j(u-Q_0u),
\gamma_{m-1-j}(\tau)
\rangle_F
\right|
\lesssim
h_T^{k+1}|u|_{H^{k+m+1}(K)}\|\tau\|_T .
$$
Together with the volume estimate, this gives
$$
|(\nabla_w^m Q_Mu-\nabla^m u,\tau)_T|
\lesssim
h_T^{k+1}|u|_{H^{k+m+1}(K)}\|\tau\|_T .
$$
Taking the supremum over $\tau\in\Sigma_k(T)$ and summing over all refined
cells proves \eqref{eq:weak-gradient-approx}.
\end{proof}

\begin{lemma}[Projected consistency]\label{lem:WG-consistency}
Let $u\in H_0^m(\Omega)$ solve
\[
(-1)^m\div^m \nabla^m u=f,
\]
and assume $\sigma:=\nabla^m u\in H^m(\Omega;\mathbb S^{d,m})$. Define
\[
\eta:=Q_\Sigma \sigma-\sigma, \qquad \Gamma_i(\sigma):=(I-Q_{F,k_i})\gamma_{m-1-i}(\sigma).
\]
Then, for any $v_h\in M_h^0$,
\begin{equation}\label{eq:projected-consistency}
(\sigma,\nabla_w^m v_h)=(f,v_0)+\mathcal E_h(u,v_h),
\end{equation}
where
\begin{equation}\label{eq:consistency-error}
\begin{aligned}
\mathcal E_h(u,v_h)
={}& \sum_{T\in\mathcal K_h^{\rm R}}(\eta,\nabla^m v_0)_T
+\sum_{j=0}^{m-1}\sum_{K\in\mathcal K_h}\sum_{F\subset\partial K}
(-1)^{m-1-j}\langle \nabla^j v_0,\Gamma_j(\sigma)\rangle_F \\
&+\sum_{j=0}^{m-1}\sum_{K\in\mathcal K_h}\sum_{F\subset\partial K}
(-1)^{m-1-j}\bigl\langle Q_{F,k_j}(v_{g,j}-\nabla^j v_0),\gamma_{m-1-j}(\eta)\bigr\rangle_F .
\end{aligned}
\end{equation}
\end{lemma}

\begin{proof}
Since $\nabla_w^m v_h\in \Sigma_{h,k}^{-1}$ and $Q_\Sigma$ is the $L^2$ projection onto $\Sigma_{h,k}^{-1}$,
\[
(\sigma,\nabla_w^m v_h)=(Q_\Sigma\sigma,\nabla_w^m v_h).
\]
By the definition of $\nabla_w^m$,
\begin{align*}
(Q_\Sigma\sigma,\nabla_w^m v_h)
={}& \sum_{T\in\mathcal K_h^{\rm R}}(Q_\Sigma\sigma,\nabla^m v_0)_T \\
&+\sum_{j=0}^{m-1}\sum_{K\in\mathcal K_h}\sum_{F\subset\partial K}
(-1)^{m-1-j}\langle Q_{F,k_j}(v_{g,j}-\nabla^j v_0),\gamma_{m-1-j}(Q_\Sigma\sigma)\rangle_F .
\end{align*}

On the other hand, integrating by parts $m$ times on each primal element $K$ gives
\begin{align*}
(f,v_0)
={}& \sum_{K\in\mathcal K_h}(\sigma,\nabla^m v_0)_K
-\sum_{j=0}^{m-1}\sum_{K\in\mathcal K_h}\sum_{F\subset\partial K}
(-1)^{m-1-j}\langle \nabla^j v_0,\gamma_{m-1-j}(\sigma)\rangle_F .
\end{align*}
Moreover,
\[
\sum_{j=0}^{m-1}\sum_{K\in\mathcal K_h}\sum_{F\subset\partial K}
(-1)^{m-1-j}\langle v_{g,j},Q_{F,k_j}\gamma_{m-1-j}(\sigma)\rangle_F=0.
\]
Since $\gamma_{m-1-j}(\sigma)=Q_{F,k_j}\gamma_{m-1-j}(\sigma)+\Gamma_j(\sigma)$, we obtain
\begin{align*}
(f,v_0)
={}& \sum_{K\in\mathcal K_h}(\sigma,\nabla^m v_0)_K \\
&+\sum_{j=0}^{m-1}\sum_{K\in\mathcal K_h}\sum_{F\subset\partial K}
(-1)^{m-1-j}\langle v_{g,j}-\nabla^j v_0,Q_{F,k_j}\gamma_{m-1-j}(\sigma)\rangle_F \\
&-\sum_{j=0}^{m-1}\sum_{K\in\mathcal K_h}\sum_{F\subset\partial K}
(-1)^{m-1-j}\langle \nabla^j v_0,\Gamma_j(\sigma)\rangle_F .
\end{align*}
Comparing the two identities yields \eqref{eq:projected-consistency}, with $\mathcal E_h(u,v_h)$ given by \eqref{eq:consistency-error}.
\end{proof}

\begin{lemma}[Residual estimate]\label{lem:Rh-approx}
Assume $\sigma\in H^m(\Omega;\mathbb S^{d,m})$. Let
$$
\sigma:=\nabla^m u,
\quad
\eta:=Q_\Sigma\sigma-\sigma ,
\quad
\Gamma_i(\sigma):=(I-Q_{F,k_i})\gamma_{m-1-i}(\sigma),
$$
and define
\begin{equation}\label{eq:Rh-def}
\mathcal R_h(u)^2
:=
\|\eta\|^2
+
\sum_{i=0}^{m-1}
\sum_{K\in\mathcal K_h}
\sum_{F\subset\partial K}
(h_F^{2i+1}
\|\gamma_i(\eta)\|_F^2 + h_F^{2m-2i-1}\|\Gamma_{i}(\sigma)\|^2_F).
\end{equation}
Then, for all $v_h\in M_h^0$,
\begin{equation}\label{eq:consistency-error-bound}
|\mathcal E_h(u,v_h)|
\lesssim
\mathcal R_h(u)\|v_h\|_{m,h}.
\end{equation}
Moreover, if $u\in H^{s^*+m}(\Omega)$, then
\begin{equation}\label{eq:Rh-approx}
\mathcal R_h(u)
\lesssim
h^{k+1}\|u\|_{s^*+m}.
\end{equation}
Here $s^*=\max\{k+1,m\}$.
\end{lemma}

\begin{proof}
By the Cauchy--Schwarz inequality, the first and the third terms are bounded by
$
\mathcal R_h(u)\|v_h\|_{m,h}.
$
Noting that $\Gamma_j(\sigma)$ is single-valued on $F$. 
\begin{align*}
  \mathcal{H}_h:=  &\sum_{j=0}^{m-1}
\sum_{K\in\mathcal K_h}
\sum_{F\subset\partial K}
(-1)^{m-1-j}
\langle
\nabla^j v_0,
\Gamma_{j}(\sigma)
\rangle_F\\  = &  \sum_{j=0}^{m-1}
\sum_{F\in \mathcal{F}_h}
(-1)^{m-1-j}
\langle
[\nabla^j v_0],
\Gamma_{j}(\sigma)
\rangle_F\\
=&\sum_{j=0}^{m-1}
\sum_{F\in \mathcal{F}_h}
(-1)^{m-1-j}
\langle
(I-Q_{F,k_j})[\nabla^j v_0],
\Gamma_{j}(\sigma)
\rangle_F.
\end{align*}
By repeating the following step up to $j=m-1$
\begin{align*} 
&\|(I-Q_{F,k_j})[\nabla^jv_0]\|_F\lesssim h_F\|[\nabla^{j+1}v_0]\|_F\\ \leq & h_F\|Q_{F,k_{j+1}}[\nabla^{j+1}v_0]\|_F + h_F\|(I-Q_{F,k_{j+1}})[\nabla^{j+1}v_0]\|_F,
\end{align*}
we have
\[
\left(\sum_{j=0}^{m-1}
\sum_{F\in \mathcal{F}_h}h_F^{2j-2m+1}\|(I-Q_{F,k_j})[\nabla^j v_0]\|^2_F\right)^{1/2}\lesssim \|v_h\|_{m,h}.
\]
Thus we have $\mathcal{H}_h\lesssim \mathcal{R}_h(u)\|v_h\|_{m,h}$, which proves \eqref{eq:consistency-error-bound}.
It remains to estimate $\mathcal R_h(u)$. Since $Q_\Sigma$ is the local $L^2$ projection onto a space containing $\mathbb P_k(T;\mathbb S^{d,m})$, the standard projection estimate gives
$$
\|\eta\|
\lesssim
h^{k+1}\|\sigma\|_{H^{k+1}(\Omega)}
=
h^{k+1}\|u\|_{H^{k+m+1}(\Omega)} .
$$
For the face terms, the scaled trace inequality and the local projection estimate give, for $F\subset\partial K$ and $0\le i\le m-1$,
$$
h_F^{i+1/2}\|\gamma_i(\eta)\|_F
\lesssim
h^{k+1}\|\sigma\|_{H^{s^*}(\omega_F)},~h_F^{m-i-1/2}\|\Gamma_{i}(\sigma)\|_F\lesssim h^{k+1}\|\sigma\|_{H^{s^*}(\omega_F)},
$$
where $\omega_F$ is the union of refined cells adjacent to $F$. Hence
$$
\sum_{i=0}^{m-1}
\sum_{K\in\mathcal K_h}
\sum_{F\subset\partial K}
h_F^{2i+1}
\|\gamma_i(\eta)\|_F^2
\lesssim
h^{2k+2}\|\sigma\|_{H^{s^*}(\Omega)}^2
=
h^{2k+2}\|u\|_{H^{s^*+m}(\Omega)}^2 .
$$
Combining the volume and face estimates proves \eqref{eq:Rh-approx}.
\end{proof}

\begin{theorem}[Energy error estimate]\label{thm:energy-error}
Let $u\in H_0^m(\Omega)\cap H^{s^*+m}(\Omega)$ and $\sigma=\nabla^m u\in H^m(\Omega;\mathbb S^{d,m})$ solve
$$
(-1)^m\operatorname{div}^m\nabla^m u=f,
$$
and let $u_h\in M_h^0$ solve \eqref{eq:WG}. Then
\begin{equation}\label{eq:energy-error}
\|Q_Mu-u_h\|_{m,h}
\lesssim
h^{k+1}\|u\|_{s^*+m}.
\end{equation}
\end{theorem}

\begin{proof}
Set $u_I:=Q_Mu$. Subtracting \eqref{eq:WG} from \eqref{eq:projected-consistency} gives, for all $v_h\in M_h^0$,
\begin{equation}\label{eq:error-equation}
(\nabla_w^m(u_I-u_h),\nabla_w^m v_h)
=
(\nabla_w^m Q_Mu-\nabla^m u,\nabla_w^m v_h)
+
\mathcal E_h(u,v_h).
\end{equation}
Taking $v_h=u_I-u_h$ and using Proposition~\ref{prop:weak-gradient-coercivity}, we obtain
\begin{align*}
\|Q_Mu-u_h\|_{m,h}
&\lesssim
\|\nabla_w^m Q_Mu-\nabla^m u\|
+
\sup_{0\ne v_h\in M_h^0}
\frac{|\mathcal E_h(u,v_h)|}{\|v_h\|_{m,h}} .
\end{align*}
The first term is bounded by Lemma~\ref{lem:weak-gradient-approx}, and the second term is bounded by Lemma~\ref{lem:Rh-approx}. This proves \eqref{eq:energy-error}.
\end{proof}

An $L^2$-error estimate can be derived by a standard Aubin--Nitsche duality argument. Since this argument is routine, we omit the details.

\begin{remark}[Low-regularity error analysis]\rm
The assumption $\sigma=\nabla^m u\in H^m(\Omega;\mathbb S^{d,m})$ in
Lemmas~\ref{lem:WG-consistency}, \ref{lem:Rh-approx}, and
Theorem~\ref{thm:energy-error} is technical and can be relaxed. By relating the present method to the $H^m$-conforming VEM in~\cite{ChenHuangWei2022}, and following the approach in~\cite{HuZhangSPP}, one can obtain error estimates under lower regularity assumptions. This low-regularity analysis is not the main focus of this paper, and the details are omitted.
\end{remark}

\section{Numerical Experiments}\label{sec:numerexperiments}

We present numerical results for the hybridized SDG method \eqref{HSDG} applied to the polyharmonic problem \eqref{intro:polyharmonic}. 
The experiments verify the predicted convergence rates and test robustness on two-dimensional polygonal meshes and three-dimensional tetrahedral meshes.

\subsection{Test problems and mesh configurations}

On the unit square $\Omega=(0,1)^2$, we use the manufactured solutions
$$
u(x,y)=\sin^m(\pi x)\sin^m(\pi y),
\qquad m=2,3.
$$
The right-hand side is chosen accordingly. The errors are measured by
$$
\|\sigma-\sigma_h\|_{0,h},\qquad
|u-u_0|_{H^m(\mathcal K_h)},\qquad
\|u-u_0\|.
$$
Here $h$ denotes the mesh size and NT denotes the number of elements.

In two dimensions, we test two families of polygonal meshes. The first consists of convex polygonal meshes generated by connecting element centroids of a triangular mesh. The second consists of concave polygonal meshes constructed from a perturbed square grid. Representative meshes are shown in Figure~\ref{fig:mesh_samples}.

For the convex polygonal meshes used in Fig \ref{fig:mesh_samples}, the number of pairwise non-parallel
face-normal directions is \(\nu=3\). Thus
\(\lfloor m/\nu\rfloor=0\) for \(m=2\), while
\(\lfloor m/\nu\rfloor=1\) for \(m=3\). For the concave polygonal meshes, we have
\(\nu=4\), and hence \(\lfloor m/\nu\rfloor=0\) for both \(m=2\) and \(m=3\).
Therefore, bubble enrichment is not needed in these two-dimensional tests, except
for the triharmonic experiment on convex polygonal meshes. The low-order cases
still require trace enrichment when \(k<m-1\).

\begin{figure}[htbp]
	\centering
	\includegraphics[width=0.20\textwidth]{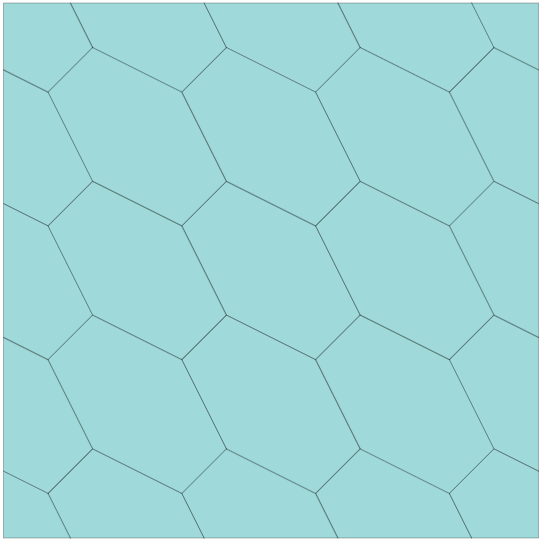}
	\includegraphics[width=0.20\textwidth]{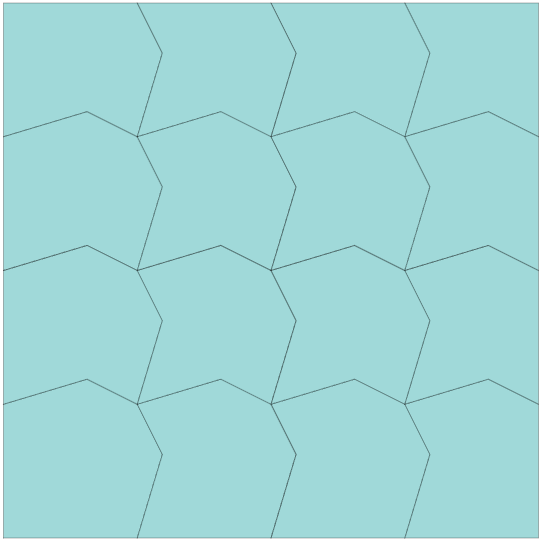}
	\caption{Representative two-dimensional polygonal meshes: convex polygonal mesh (left) and concave polygonal mesh (right).}
	\label{fig:mesh_samples}
\end{figure}

\subsection{Two-dimensional biharmonic problem}

Tables~\ref{tab:biharmonic_polygonal}--\ref{tab:biharmonic_concave} report the errors for the biharmonic problem, corresponding to $m=2$, on convex and concave polygonal meshes. The energy-type errors converge with order $k+1$, and the $L^2$ error $\|u-u_0\|$ shows higher-order convergence.

\begin{table}[htbp]
	\centering
	\caption{Biharmonic problem ($m=2, d=2$) on convex polygonal meshes.}
	\label{tab:biharmonic_polygonal}
	\small
	\begin{tabular}{ccccccccc}
		\toprule
		$k$ & $h$ & NT &
		$\|\sigma-\sigma_h\|_{0,h}$ & Order &
		$|u-u_0|_{H^2(\mathcal K_h)}$ & Order & $\|u-u_0\|$ & Order\\
		\midrule
		0 &1.250e-01    &81     &3.77e+00   &-      &1.85e+01   &-      &3.53e-03   &-\\
		&6.250e-02    &289    &1.90e+00   &0.99   &9.51e+00   &0.96   &7.33e-04   &2.27\\
		&3.125e-02    &1089   &9.54e-01   &0.99   &4.79e+00   &0.99   &1.73e-04   &2.08\\
		&1.562e-02    &4225   &4.78e-01   &1.00   &2.40e+00   &1.00   &4.27e-05   &2.02\\
		&7.812e-03    &16641  &2.39e-01   &1.00   &1.20e+00   &1.00   &1.06e-05   &2.00\\
		\midrule
		1 &1.250e-01 & 81 & 9.28e-01 & - & 6.28e-01 & - & 1.93e-04 & - \\
		&6.250e-02 & 289 & 2.38e-01 & 1.97 & 1.66e-01 & 1.92 & 1.26e-05 & 3.93 \\
		&3.125e-02 & 1089 & 5.96e-02 & 1.99 & 4.26e-02 & 1.96 & 8.06e-07 & 3.97 \\
		&1.562e-02 & 4225 & 1.49e-02 & 2.00 & 1.08e-02 & 1.98 & 5.09e-08 & 3.98 \\
		&7.812e-03 & 16641 & 3.73e-03 & 2.00 & 2.71e-03 & 1.99 & 3.21e-09 & 3.98 \\
		\bottomrule
	\end{tabular}
\end{table}

\begin{table}[htbp]
	\centering
	\caption{Biharmonic problem ($m=2, d=2$) on concave polygonal meshes.}
	\label{tab:biharmonic_concave}
	\small
	\begin{tabular}{ccccccccc}
		\toprule
		$k$ & $h$ & NT &
		$\|\sigma-\sigma_h\|_{0,h}$ & Order &
		$|u-u_0|_{H^2(\mathcal K_h)}$  & Order & $\|u-u_0\|$ & Order\\
		\midrule
		0 & 1.362e-01 & 64 & 2.00e+01 & - & 3.63e+00 & - & 1.89e-03 & - \\
		& 6.811e-02 & 256 & 1.01e+01 & 0.98 & 1.84e+00 & 0.98 & 2.54e-04 & 2.89 \\
		& 3.405e-02 & 1024 & 5.08e+00 & 0.99 & 9.25e-01 & 0.99 & 7.00e-05 & 1.86 \\
		& 1.703e-02 & 4096 & 2.54e+00 & 1.00 & 4.63e-01 & 1.00 & 2.07e-05 & 1.76 \\
		& 8.513e-03 & 16384 & 1.27e+00 & 1.00 & 2.32e-01 & 1.00 & 5.67e-06 & 1.87 \\
		\midrule
		1 & 1.362e-01 & 64 & 1.12e+00 & - & 6.80e-01 & - & 1.98e-04 & - \\
		& 6.811e-02 & 256 & 2.69e-01 & 2.05 & 1.73e-01 & 1.98 & 1.29e-05 & 3.94 \\
		& 3.405e-02 & 1024 & 6.55e-02 & 2.04 & 4.33e-02 & 1.99 & 8.27e-07 & 3.97 \\
		& 1.703e-02 & 4096 & 1.61e-02 & 2.02 & 1.08e-02 & 2.00 & 5.26e-08 & 3.97 \\
		& 8.513e-03 & 16384 & 4.01e-03 & 2.01 & 2.71e-03 & 2.00 & 5.64e-09 & 3.22 \\
		\bottomrule
	\end{tabular}
\end{table}

\subsection{Two-dimensional triharmonic problem}

Tables~\ref{tab:triharmonic_polygonal}--\ref{tab:triharmonic_concave} report the errors for the triharmonic problem, corresponding to $m=3$, on convex and concave polygonal meshes. 
These tests include trace enrichment for $k=0,1$, but no bubble enrichment is required because $\lfloor m/\nu\rfloor=0$ for both mesh families. The energy-type errors converge with order $k+1$, and the $L^2$ error $\|u-u_0\|$ shows higher-order convergence.

\begin{table}[htbp]
	\centering
	\caption{Triharmonic problem ($m=3, d=2$) on convex polygonal meshes.}
	\label{tab:triharmonic_polygonal}
	\small
	\begin{tabular}{ccccccccc}
		\toprule
		$k$ & $h$ & NT &
		$\|\sigma-\sigma_h\|_{0,h}$ & Order &
		$|u-u_0|_{H^3(\mathcal K_h)}$ & Order &
		$\|u-u_0\|$ & Order \\
		\midrule
		0 & 2.500e-01 & 25 & 8.34e+02 & - & 1.01e+02 & - & 1.13e-02 & - \\
		& 1.250e-01 & 81 & 3.92e+02 & 1.09 & 5.37e+01 & 0.91 & 1.24e-03 & 3.18 \\
		& 6.250e-02 & 289 & 1.59e+02 & 1.30 & 2.73e+01 & 0.97 & 3.66e-04 & 1.76 \\
		& 3.125e-02 & 1089 & 4.96e+01 & 1.68 & 1.38e+01 & 0.99 & 1.13e-04 & 1.69 \\
		& 1.562e-02 & 4225 & 1.54e+01 & 1.68 & 6.91e+00 & 1.00 & 3.24e-05 & 1.81 \\
		\midrule
		1 & 5.000e-01 & 9 & 1.55e+03 & - & 1.44e+02 & - & 7.58e-02 & - \\
		& 2.500e-01 & 25 & 3.83e+02 & 2.01 & 5.32e+01 & 1.43 & 3.02e-03 & 4.65 \\
		& 1.250e-01 & 81 & 8.94e+01 & 2.10 & 1.45e+01 & 1.88 & 1.39e-04 & 4.45 \\
		& 6.250e-02 & 289 & 2.71e+01 & 1.72 & 3.86e+00 & 1.91 & 9.22e-06 & 3.91 \\
		& 3.125e-02 & 1089 & 7.48e+00 & 1.86 & 9.88e-01 & 1.96 & 6.29e-07 & 3.87 \\
		\midrule
		2 & 5.000e-01 & 9 & 2.06e+03 & - & 8.51e+01 & - & 9.98e-03 & - \\
		& 2.500e-01 & 25 & 4.25e+02 & 2.28 & 1.77e+01 & 2.26 & 3.45e-04 & 4.85 \\
		& 1.250e-01 & 81 & 5.51e+01 & 2.95 & 2.73e+00 & 2.70 & 9.34e-06 & 5.21 \\
		& 6.250e-02 & 289 & 6.84e+00 & 3.01 & 3.66e-01 & 2.90 & 1.81e-07 & 5.69 \\
		& 3.125e-02 & 1089 & 8.53e-01 & 3.00 & 4.70e-02 & 2.96 & 7.01e-09 & 4.69 \\
		\bottomrule
	\end{tabular}
\end{table}

\begin{table}[htbp]
	\centering
	\caption{Triharmonic problem ($m=3, d=2$) on concave polygonal meshes.}
	\label{tab:triharmonic_concave}
	\small
	\begin{tabular}{ccccccccc}
		\toprule
		$k$ & $h$ & NT &
		$\|\sigma-\sigma_h\|_{0,h}$ & Order &
		$|u-u_0|_{H^3(\mathcal K_h)}$ & Order & $\|u-u_0\|$ & Order\\
		\midrule
		0 &2.724e-01    &16     &1.48e+03   &-   &1.04e+02   &-   &2.37e-02   &-\\
		&1.362e-01    &64     &6.16e+02   &1.27   &5.47e+01   &0.92   &5.21e-03   &2.19\\
		&6.811e-02    &256    &2.08e+02   &1.57   &2.78e+01   &0.98   &1.16e-03   &2.17\\
		&3.405e-02    &1024   &7.58e+01   &1.45   &1.40e+01   &0.99   &2.83e-04   &2.03\\
		&1.703e-02    &4096   &3.29e+01   &1.20   &6.99e+00   &1.00   &7.86e-05   &1.85\\
		\midrule
		1 &5.449e-01    &4      &2.34e+03   &-   &1.49e+02   &-   &2.94e-01   &-\\
		&2.724e-01    &16     &6.28e+02   &1.90   &5.20e+01   &1.52   &3.72e-03   &6.30\\
		&1.362e-01    &64     &1.31e+02   &2.26   &1.51e+01   &1.78   &1.90e-04   &4.29\\
		&6.811e-02    &256    &2.85e+01   &2.20   &3.93e+00   &1.94   &1.78e-05   &3.42\\
		&3.405e-02    &1024   &6.85e+00   &2.06   &9.92e-01   &1.99   &1.42e-06   &3.64\\
		\midrule
		2 &1.000e+00    &1      &2.74e+04   &-      &1.82e+02   &-      &2.40e-01   &-\\
		&5.449e-01    &4      &5.13e+03   &2.76   &1.01e+02   &0.98   &7.82e-02   &1.85\\
		&2.724e-01    &16     &6.25e+02   &3.04   &1.83e+01   &2.46   &6.44e-04   &6.92\\
		&1.362e-01    &64     &7.79e+01   &3.01   &2.62e+00   &2.80   &1.13e-05   &5.83\\
		&6.811e-02    &256    &1.04e+01   &2.90   &3.35e-01   &2.97   &1.88e-07   &5.91\\
		\bottomrule
	\end{tabular}
\end{table}

\subsection{Three-dimensional biharmonic problem}

We also test the biharmonic problem in three dimensions on the unit cube $\Omega=(0,1)^3$, with exact solution
$$
u(x,y,z)=\sin^2(\pi x)\sin^2(\pi y)\sin^2(\pi z).
$$
The domain is discretized by tetrahedral meshes. Although tetrahedra are simplicial elements rather than general polyhedra, they are polytopes and provide a useful three-dimensional test of the dimension-independent construction. For a tetrahedron, $\nu=4$, and hence
$
\left\lfloor\frac{m}{\nu}\right\rfloor
=
\left\lfloor\frac{2}{4}\right\rfloor
=0.
$
Thus no bubble enrichment is needed in this test. Table~\ref{tab:biharmonic_3D} shows that the method retains the predicted convergence behavior in three dimensions: the energy-type errors converge with order $k+1$, and the $L^2$ error shows higher-order convergence.

\begin{table}[htbp]
	\centering
	\caption{Biharmonic problem ($m=2, d=3$) on tetrahedral meshes.}
	\label{tab:biharmonic_3D}
	\small
	\begin{tabular}{ccccccccc}
		\toprule
		$k$ & $h$ & NT &
		$\|\sigma-\sigma_h\|_{0,h}$ & Order &
		$|u-u_0|_{H^2(\mathcal K_h)}$ & Order &
		$\|u-u_0\|$ & Order \\
		\midrule
		0 & 1.376e-01 & 384 & 1.43e+01 & - & 6.82e+00 & - & 4.40e-02 & - \\
		& 9.172e-02 & 1296 & 7.95e+00 & 1.45 & 4.46e+00 & 1.05 & 1.89e-02 & 2.09 \\
		& 6.879e-02 & 3072 & 5.47e+00 & 1.30 & 3.32e+00 & 1.03 & 1.05e-02 & 2.03 \\
		& 5.503e-02 & 6000 & 4.18e+00 & 1.21 & 2.64e+00 & 1.02 & 6.73e-03 & 2.01 \\
		\midrule
		1 & 1.376e-01 & 384 & 9.86e+00 & - & 1.50e+00 & - & 3.93e-03 & - \\
		& 9.172e-02 & 1296 & 4.45e+00 & 1.96 & 6.80e-01 & 1.95 & 8.04e-04 & 3.91 \\
		& 6.879e-02 & 3072 & 2.52e+00 & 1.98 & 3.85e-01 & 1.98 & 2.57e-04 & 3.96 \\
		& 5.503e-02 & 6000 & 1.62e+00 & 1.99 & 2.47e-01 & 1.99 & 1.06e-04 & 3.98 \\
		\bottomrule
	\end{tabular}
\end{table}

\bibliographystyle{abbrv}
 \bibliography{./paper.bib,./references.bib}
\end{document}

%% file: mysetting.tex
\usepackage{times,amsmath,amsbsy,amssymb,amscd,mathrsfs}
\usepackage{slashbox}\usepackage{booktabs,graphicx,subfigure,epstopdf,wrapfig,chemarrow}

\usepackage{algorithm2e} 
\usepackage{multicol,multirow}
\usepackage{mathtools}
\usepackage[usenames,dvipsnames,svgnames,table]{xcolor}
\usepackage[all]{xy}
\usepackage{wrapfig}
\usepackage{tcolorbox}
\usepackage{stmaryrd}

\usepackage{tikz,tikz-cd}
\usepackage[utf8]{inputenc}
\usepackage{pgfplots} 
\usepackage{pgfgantt}
\usepackage{pdflscape}
\pgfplotsset{compat=newest} 
\pgfplotsset{plot coordinates/math parser=false}
\newlength\fwidth

\usepackage[numbered]{mcode}

\definecolor{myBlue}{rgb}{0.0,0.0,0.55}
\usepackage[pdftex,colorlinks=true,citecolor=myBlue,linkcolor=myBlue]{hyperref}

\usepackage[hyperpageref]{backref}
\usepackage{circledsteps}

\newcommand{\LC}[1]{\textcolor{red}{#1}}

\usepackage{comment,enumerate,multicol,xspace}

  \newcounter{mnote}
  \let\oldmarginpar\marginpar
    \renewcommand\marginpar[1]{\-\oldmarginpar[\raggedleft\footnotesize #1]%
    {\raggedright\footnotesize #1}}

\newcommand{\breakline}{
\begin{center}
------------------------------------------------------------------------------------------------------------
\end{center}
}

\newtheorem{theorem}{Theorem}[section]
\newtheorem{lemma}[theorem]{Lemma}

\newtheorem{proposition}[theorem]{Proposition}
\newtheorem{definition}[theorem]{Definition}
\newtheorem{example}[theorem]{Example}

\newtheorem{remark}[theorem]{Remark}

\renewcommand{\div}{\operatorname{div}}

\newcommand{\vertiii}[1]{{\left\vert\kern-0.25ex\left\vert\kern-0.25ex\left\vert #1 
    \right\vert\kern-0.25ex\right\vert\kern-0.25ex\right\vert}}